\documentclass[a4paper,11pt]{article}
\usepackage[margin=1truein]{geometry}

\usepackage[T1]{fontenc}
\usepackage{graphicx} %
\usepackage{amsmath,amsthm,amssymb}
\usepackage{mathtools}
\usepackage{thmtools}
\usepackage{upref}
\usepackage{algorithm}
\usepackage{booktabs}
\usepackage{enumitem}
\usepackage{comment}

\usepackage[%
    bookmarksnumbered,
    hypertexnames=false,
    pdfdisplaydoctitle,
    pdfusetitle,
    unicode
]{hyperref}

\usepackage[color={red!100!green!33},textsize=scriptsize]{todonotes}

\usepackage[noend]{algpseudocode}
\algnewcommand{\algorithmicinput}{\textbf{Input:}}
\algnewcommand{\Input}{\item[\algorithmicinput]}
\algnewcommand{\algorithmicoutput}{\textbf{Output:}}
\algnewcommand{\Output}{\item[\algorithmicoutput]}

\usepackage[capitalise,noabbrev]{cleveref}

\crefname{line}{Line}{Lines}
\Crefname{line}{Line}{Lines}

\usepackage{autonum}
\makeatletter
\autonum@generatePatchedReferenceCSL{Cref}
\makeatother

\usepackage{mleftright}
\mleftright

\usepackage{etoolbox}
\cslet{blx@noerroretextools}\empty%
\usepackage[%
    doi=true,
    isbn=false,
    articlein=false,
    url=false,
    doi=false,
    eprint=false,
    giveninits=true,
    maxbibnames=99,
    minalphanames=4,
    maxalphanames=4,
    natbib,
    sorting=nyt,
    style=ext-alphabetic
]{biblatex}

\DeclareFieldFormat[article,inbook,incollection,inproceedings,patent,thesis,unpublished]{titlecase:title}{\MakeSentenceCase*{#1}} %

\declaretheorem[style=plain,numberwithin=section,name=Theorem]{theorem}
\declaretheorem[style=plain,sibling=theorem,name=Lemma]{lemma}

\declaretheorem[style=plain,sibling=theorem,name=Corollary]{corollary}

\declaretheorem[style=definition,sibling=theorem,name=Example,qed=$\blacksquare$]{example}

\declaretheorem[style=definition,sibling=theorem,name=Remark,qed=$\blacksquare$]{remark}

\crefname{theorem}{Theorem}{Theorems}
\crefname{proposition}{Proposition}{Propositions}
\crefname{lemma}{Lemma}{Lemmas}
\crefname{exmp}{Example}{Examples}
\crefname{corollary}{Corollary}{Corollarys}
\crefname{claim}{Claim}{Claims}
\crefname{remark}{Remark}{Remarks}
\crefname{section}{Section}{Sections}

\numberwithin{equation}{section}

\newcommand{\C}{\mathbb{C}}
\newcommand{\R}{\mathbb{R}}
\newcommand{\Q}{\mathbb{Q}}
\newcommand{\Z}{\mathbb{Z}}

\newcommand{\caA}{\mathcal{A}}
\newcommand{\caL}{\mathcal{L}}

\newcommand{\be}{\mathbf{e}}

\newcommand{\bu}{\mathbf{u}}

\newcommand{\eps}{\varepsilon}

\newcommand{\ones}{\mathbf{1}}
\newcommand{\B}{\mathrm{B}}
\newcommand{\BB}{\mathbf{B}}
\newcommand{\M}{\mathbf{M}}
\DeclareMathOperator{\ds}{\mathsf{ds}}
\DeclareMathOperator{\GL}{GL}

\DeclareMathOperator{\Mat}{Mat}

\DeclareMathOperator{\PD}{PD}
\DeclareMathOperator{\Rep}{Rep}
\DeclareMathOperator{\In}{In}
\DeclareMathOperator{\Out}{Out}
\DeclareMathOperator{\tr}{tr}
\DeclareMathOperator{\im}{im}
\DeclareMathOperator{\supp}{supp}
\DeclareMathOperator{\Diag}{Diag}
\DeclareMathOperator{\dimv}{\underline{dim}}

\DeclareMathOperator{\poly}{poly}

\DeclareMathOperator{\conv}{conv}

\DeclareMathOperator{\ncrank}{nc-rank}
\DeclarePairedDelimiter{\norm}{\lVert}{\rVert}
\DeclarePairedDelimiter{\abs}{\lvert}{\rvert}
\DeclarePairedDelimiter{\spnn}{\langle}{\rangle}
\usepackage{xcolor}

\title{Algorithmic aspects of semistability of quiver representations}
\author{Yuni Iwamasa\thanks{Graduate School of Informatics, Kyoto University, Kyoto, 606-8501, Japan. Email: \texttt{iwamasa@i.kyoto-u.ac.jp}}%
\and Taihei Oki\thanks{Institute for Chemical Reaction Design and Discovery (ICReDD), Hokkaido University, Sapporo, 001-0021, Japan. Email: \texttt{oki@icredd.hokudai.ac.jp}}%
\and Tasuku Soma\thanks{The Institute of Statistical Mathematics, Tokyo, 190-8562, Japan. Email: \texttt{soma@ism.ac.jp}}}
\date{\today}

\begin{document}

\maketitle

\begin{abstract}
We study the semistability of quiver representations from an algorithmic perspective.
We present efficient algorithms for several fundamental computational problems on the semistability of quiver representations: deciding the semistability and $\sigma$-semistability, finding the maximizers of King's criterion, and computing the Harder--Narasimhan filtration.
We also investigate a class of polyhedral cones defined by the linear system in King's criterion, which we refer to as King cones. 
For rank-one representations, we demonstrate that these King cones can be encoded by submodular flow polytopes, enabling us to decide the $\sigma$-semistability in strongly polynomial time.
Our approach employs submodularity in quiver representations, which may be of independent interest.
\end{abstract}
\allowdisplaybreaks

\section{Introduction and our contribution}\label{sec:introduction}

\emph{Quiver representation} is a simple generalization of matrices that has led to surprisingly deep extensions of various results in linear algebra~\citep{Derksen2017book}. 
In this paper, we study the \emph{semistability} of quiver representations, which is a central concept in the \emph{geometric invariant theory} (GIT), from an algorithmic perspective.
The semistability of quiver representations appears in \emph{operator scaling}~\citep{Gurvits2004,Garg2019,Franks2018,Burgisser2018a,Franks2023}, Brascamp--Lieb (BL) inequality~\citep{Garg2018}, Tyler's M-estimator~\citep{Franks2020}, and scatter estimation of structured normal models~\citep{Amendola2021}, which have attracted considerable attention in theoretical computer science owing to their connection to the noncommutative Edmonds' problem, algebraic complexity theory, and submodular optimization~\citep{Mulmuley2017,Ivanyos2018,Burgisser2019,Hamada2021}.
The goal of this paper is to provide efficient algorithms for various fundamental computational problems on the semistability of quiver representations.
In the following, we describe the problems more formally and present our results.

\subsection{Semistability of quiver representations}\label{sec:semistability}
Here, we present the formal definition of a quiver representation.
We follow the terminologies in \citep{Derksen2017book,Burgisser2019}.
Let $Q = (Q_0, Q_1)$ be a quiver with a vertex set $Q_0$ and an arc set $Q_1$.
In this paper, we consider only acyclic quivers except in \cref{sec:general-ss}.
For each arc $a \in Q_1$, we denote the tail and head of $a$ by $ta$ and $ha$, respectively.
A \emph{representation} $V$ of $Q$ consists of complex vector spaces $V(i)$ for vertex $i \in Q_0$ and linear maps $V(a): V(ta) \to V(ha)$ for arc $a \in Q_1$.
A \emph{subrepresentation} $W$ of $V$ is a representation of the same quiver such that $W(i) \leq V(i)$ for $i \in Q_0$, and $W(a) = V(a)|_{W(ta)}$ and $\im W(a) \leq W(ha)$ for $a \in Q_1$.
The vector of dimensions of $V(i)$ is called the \emph{dimension vector} of the representation, denoted by $\dimv V$.
We call the vector space of all representations of $Q$ with the dimension vector $\alpha$ the \emph{representation space} of $Q$ with dimension vector $\alpha$, which we denote by $\Rep(Q, \alpha)$.
After fixing the dimension vector $\alpha$ and a basis of each $V(i)$, we can represent $V(a)$ as an $\alpha(ha) \times \alpha(ta)$ matrix.
Therefore, the representation space can be identified as
\begin{align}
    \Rep(Q, \alpha) = \bigoplus_{a \in Q_1} \Mat(\alpha(ha), \alpha(ta)),
\end{align}
where $\Mat(m,n)$ denotes the space of $m \times n$ complex matrices.

Fix a dimension vector $\alpha$.
Let
\begin{align}
    \GL(Q, \alpha) \coloneqq \prod_{i \in Q_0} \GL(\alpha(i)),
\end{align}
where $\GL(n)$ denotes the general linear group of degree $n$.
Then, $\GL(Q, \alpha)$ acts on the representation space by a change of basis:
\begin{align}
    g \cdot V \coloneqq (g_{ha} V(a) g_{ta}^{-1})_{a \in Q_1}.
\end{align} 
Note that this is a left action, i.e., $(gh) \cdot V = g \cdot (h \cdot V)$ for $g, h \in \GL(Q, \alpha)$.
We say that a representation $V$ is \emph{semistable} under the $\GL(Q,\alpha)$-action if the orbit closure of $V$ does not contain the origin, i.e.,
\begin{align}
    \inf_{g \in \GL(Q, \alpha)} \sum_{a \in Q_1} \norm{g_{ha} V(a) g_{ta}^{-1}}_{\mathrm{F}}^2 > 0.
\end{align}
Otherwise, $V$ is said to be \emph{unstable}.
The set of all unstable representations is called the \emph{null-cone} of the $\GL(Q, \alpha)$-action.
It is easy to see that any representation is unstable under the $\GL(Q, \alpha)$-action if $Q$ is acyclic.\footnote{The readers may wonder whether we can check the semistability (under $\GL(Q, \alpha)$-action) of quiver representations of \emph{cyclic} quivers. We will address this point later.}

However, the semistability of quivers under subgroups of $\GL(Q, \alpha)$ turns out to be more intricate.
Let $\sigma \in \Z^{Q_0}$ be an integer vector on $Q_0$, which we call a \emph{weight}.
Let $\chi_\sigma$ be the corresponding multiplicative character of $\GL(Q, \alpha)$, i.e.,
\[
    \chi_\sigma(g) = \prod_{i \in Q_0} \det(g_{i})^{\sigma(i)}.   
\]
Note that $\chi_\sigma$ is a one-dimensional representation of $\GL(Q, \alpha)$; $\GL(Q, \alpha)$ acts on $\C$ by $g \cdot x \coloneqq \chi_\sigma(g) x$.
A representation $V$ is said to be \emph{$\sigma$-semistable} if the orbit closure of $(V, 1) \in \Rep(Q, \alpha) \oplus \C$ under the $\GL(Q, \alpha)$ action does not contain the origin, i.e.,
\begin{align}
    \inf_{g \in \GL(Q, \alpha)} \left(\sum_{a \in Q_1} \norm{g_{ha} V(a) g_{ta}^{-1}}_{\mathrm{F}}^2 + \abs{\chi_\sigma(g)}^2 \right) > 0.
\end{align}
It turns out that checking the $\sigma$-semistability of a quiver representation includes operator scaling (noncommutative rank computation) and the membership problem of the BL polytopes.
We will see these examples in the following sections.

Our first result is a deterministic algorithm that, given a quiver representation $V$ and weight $\sigma$, decides whether the representation is $\sigma$-semistable in time polynomial in the \emph{bit complexity} of $V$ and absolute values of the entries of $\sigma$.
Let $\alpha(Q_0) \coloneqq \sum_{i \in Q_0} \alpha(i)$.

\begin{theorem}[informal version of \Cref{thm:ss}]
    Let $Q$ be an acyclic quiver, $V$ a representation of $Q$, and $\sigma$ a weight.
    There is a deterministic algorithm that decides the $\sigma$-semistability of $V$ in time polynomial in the size of $Q$, $\alpha(Q_0)$, bit complexity of $V$, and absolute values of the entries of $\sigma$.
\end{theorem}

This improves the previous result~\cite{Huszar2021} which runs in time polynomial in the number of \emph{paths} in $Q$, which can be exponential in the size of $Q$.
Furthermore, if the absolute value of the entries of $\sigma$ is constant, our algorithm runs in polynomial time.
This includes the known result for operator scaling~\cite{Garg2019}.

\subsection{King's criterion}
\citet{King1994} showed the following characterization of $\sigma$-semistability, which is known as \emph{King's criterion}.
Let $\sigma(\alpha) \coloneqq \sum_{i \in Q_0} \sigma(i) \alpha(i)$ for a dimension vector $\alpha$.
Then, a representation $V$ is $\sigma$-semistable if and only if $\sigma(\dimv V) = 0$ and  $\sigma(\dimv W) \leq 0$ for any subrepresentation $W$ of $V$.

King's criterion is a common generalization of the noncommutative rank (nc-rank) computation and membership problem of BL polytopes.

\begin{example}[nc-rank]\label{ex:nc-rank}
    Let $Q$ be the generalized Kronecker quiver with $m$ parallel arcs, $\alpha = (n, n)$, and $\sigma = (1,-1)$; see~\Cref{fig:Kr-star-quivers}.
    Any representation $V$ of $Q$ with the dimension vector $\alpha$ can be regarded as an $n \times n$ linear matrix $A = \sum_{a=1}^m x_a V(a)$, where $x_a$ is an indeterminate.
    A subrepresentation $W$ of $V$ is determined by a pair of subspaces $(W(1), W(2))$ such that $\sum_{a=1}^m V(a) W(1) \leq W(2)$.
    Then, King's criterion reads that $V$ is $\sigma$-semistable if and only if $\dim U - \dim(\sum_{a=1}^m V(a) U) \leq 0$ for any subspace $U \leq \C^n$, which is equivalent to that $A$ is \emph{nc-nonsingular}. 
    More generally, the nc-rank of $A$ is equal to the minimum of $n + \dim U - \dim(\sum_{a=1}^m V(a) U)$ over all subspaces $U \leq \C^n$~\citep{Fortin2004}.
\end{example}

\begin{example}[BL polytope]\label{ex:BL polytope}
    Let $Q$ be a star quiver with $m$ leaves.
    We assume that $Q_0 = \{0, 1, \dots, m\}$ and $0$ is the root; see~\Cref{fig:Kr-star-quivers}.
    Let $\alpha = (n, n_1, \dots, n_m)$ and $\sigma = (d, -c_1, \dots, -c_m)$ for positive integers $d, c_1, \dots, c_m$.
    A real representation $V$ of $Q$ with the dimension vector $\alpha$ can be regarded as a tuple of the matrices $(B_1, \dots, B_m)$, where $B_i$ is an $n_i \times n$ matrix.
    Again, a subrepresentation $W$ is an $(m+1)$-tuple of the subspaces $(W(0), W(1), \dots, W(m))$ such that $B_i W(0) \leq W(i)$ for $i \in [m]$.
    King's criterion reads that $V$ is $\sigma$-semistable if and only if $dn = \sum_{i=1}^m c_in_i$ and $dn \dim W(0) - \sum_{i=1}^m c_i n_i \dim (B_i W(0)) \leq 0$ for any subspace $W(0)$.
    This is equivalent to that $p = (c_1/d, \dots, c_m/d)$ is in the BL polytope of linear operators $B_1, \dots, B_m$~\citep{Bennett2008}.
\end{example}

\begin{figure}
    \centering
    \includegraphics[width=0.7\textwidth]{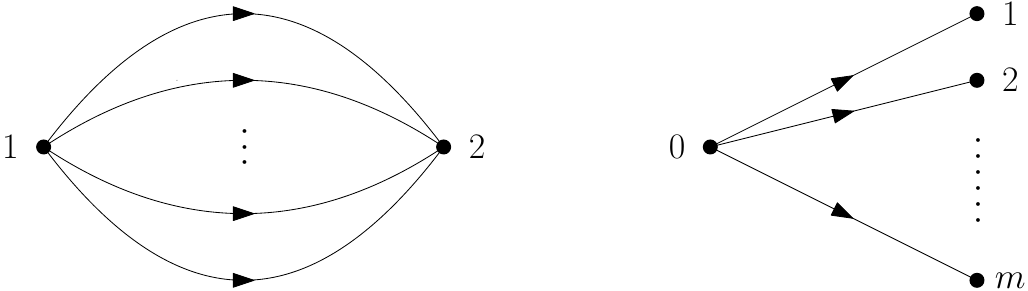}
    \caption{Generalized Kronecker quiver (left) and star quiver (right).\label{fig:Kr-star-quivers}}
\end{figure}

We study the following optimization problem: given a quiver representation $V$ and weight $\sigma$, find a subrepresentation $W$ of $V$ that maximizes $\sigma(\dimv W)$.
In the case of the nc-rank, such a subrepresentation corresponds to a subspace $U$ that maximizes $\dim U - \dim(\sum_{a=1}^m A_i U)$.
Such a subspace is called a \emph{shrunk subspace} and can be regarded as a certificate of the nc-rank~\citep{Ivanyos2018,Franks2023}.
In the case of the BL polytopes, the problem corresponds to \emph{separation} for the BL polytope~\citep{Garg2018}.

King's criterion can be regarded as \emph{maximizing a modular function over the modular lattice of subrepresentations}.
For any subrepresentations $W_1, W_2$ of $V$, define the subrepresentations $W_1 + W_2$ and $W_1 \cap W_2$ as follows.
For each $i \in Q_0$, 
\begin{align}
    (W_1 + W_2)(i) \coloneqq W_1(i) + W_2(i), \quad
    (W_1 \cap W_2)(i) \coloneqq W_1(i) \cap W_2(i),
\end{align}
where the addition and intersection on the right-hand side are those of the vector spaces.
Furthermore, the linear map of $a \in Q_1$ in $W_1 + W_2$ (resp.~$W_1 \cap W_2$) is defined as the restriction of $V(a)$ to $(W_1 + W_2)(ta)$ (resp.~$(W_1 \cap W_2)(ta)$).
Then, $W_1 + W_2$ and $W_1 \cap W_2$ are indeed subrepresentations of $V$.
Thus, the subrepresentations of $V$ form a modular lattice. 
Furthermore, the function $f(W) \coloneqq \sigma(\dimv W)$ is a modular function, i.e., for any subrepresentations $W_1, W_2$ of $V$,
\begin{align}
    f(W_1) + f(W_2) = f(W_1 + W_2) + f(W_1 \cap W_2).
\end{align}
Thus, subrepresentations maximizing $f$ form a sublattice, and there is a unique inclusion-wise minimum maximizer.
Our second result is a deterministic algorithm to find such a maximizer of King's criterion.

\begin{theorem}[informal version of \Cref{thm:King-maximizer}]
    Let $Q$ be an acyclic quiver, $V$ a representation of $Q$, and $\sigma$ a weight.
    There is a deterministic algorithm that finds the inclusion-wise minimum maximizer $W$ of King's criterion in time polynomial in the size of $Q$, $\alpha(Q_0)$, bit complexity of $V$, and absolute values of the entries of $\sigma$.
\end{theorem}

King's criterion was originally proved using the \emph{Hilbert-Mumford criterion} (see, e.g.,~\cite[Sections 9.6 and 9.8]{Derksen2017book}), a fundamental result in the GIT.
We provide an alternative elementary proof in \Cref{sec:King} for the sake of completeness.

\subsection{Harder-Narasimhan filtration}\label{subsec:intro:HN-filtration}
We use the algorithm for finding the maximizers of King's criterion to devise an algorithm for finding the \emph{Harder-Narasimhan (HN) filtration}~\cite{Harder1975,Hille2002} of a quiver representation.
Roughly speaking, the HN-filtration decomposes a quiver representation into the direct sum of smaller representations.

More precisely, let $\sigma \in \Z^{Q_0}$ be a weight and $\tau \in \Z_+^{Q_0}$ a strictly monotone weight, i.e., a nonnegative weight such that $\tau(\dimv W) > 0$ if $W \neq \{0\}$.
Here, $\{0\}$ denotes the \emph{zero representation}, which is the representation whose dimension vector is the zero vector.
We define the \emph{slope} of a quiver nonzero representation $V$ as $\mu(V) = \sigma(\dimv V)/\tau(\dimv V)$.
We say that $V$ is \emph{$\mu$-semistable}\footnote{It is also called \emph{$(\sigma:\tau)$-semistability} in the literature.} if $\mu(W) \leq \mu(V)$ for any nonzero subrepresentation $W$ of $V$.
The HN-filtration theorem states that for any quiver representation $V$, there exists a unique filtration $\{0\} = W_0 < W_1 < \cdots < W_k = V$ such that (i) $\mu(W_i/W_{i-1}) > \mu(W_{i+1}/ W_i)$ for $i \in [k-1]$ and (ii) $W_i/W_{i-1}$ is $\mu$-semistable.
Here,  $W_i < W_{i+1}$ means that $W_i$ is a subrepresentation of $W_{i+1}$ with $W_i \ne W_{i+1}$, and $W_i/W_{i-1}$ is a representation of $Q$ such that $(W_i/W_{i-1})(j)$ is the quotient space $W_i(j)/W_{i-1}(j)$ for $j \in Q_0$ and $(W_i/W_{i-1})(a)$ is the corresponding quotient linear map of $W_i(a)$ for $a \in Q_1$.
We note that semistability with respect to a slope can be reduced to that for a weight; see~\Cref{lem:slope-to-weight}.

Our third result is a deterministic algorithm for finding the HN-filtration.

\begin{theorem}[informal version of \Cref{thm:HN}]
    Let $Q$ be an acyclic quiver, $V$ a representation of $Q$, and $\mu = \sigma/\tau$ a slope.
    There is a deterministic algorithm that finds the HN-filtration of $V$ with respect to $\mu$ in time polynomial in the size of $Q$, $\alpha(Q_0)$, bit complexity of $V$, and absolute values of the entries of $\sigma$ and $\tau$.
\end{theorem}

This result improves a recent result~\cite{Cheng2024} which runs in time polynomial in the number of paths in $Q$.

Recently, \citet{Hirai2024} introduced the \emph{coarse Dulmage-Mendelsohn (DM) decomposition} of a linear matrix, generalizing the classic DM-decomposition of a bipartite graph.
They showed that a natural gradient flow of operator scaling converges to the coarse DM-decomposition.
However, their result did not provide an efficient algorithm to compute the coarse DM-decomposition because the gradient flow may take exponential time to converge.
We show that the coarse DM-decomposition is indeed a special case of the HN-filtration for the generalized Kronecker quiver.
Since the absolute values of the weights involved for this special case are polynomially bounded, our algorithm finds the coarse DM-decomposition in polynomial time; see~\Cref{subsec:coarse-DM}.

\subsection{King's polyhedral cone, rank-one representations, and submodular flow}\label{subsec:Intro:rank-one}

Motivated by King's criterion, we investigate a polyhedral cone that is the set of $\sigma \in \R^{Q_0}$ satisfying $\sigma(\dimv V) = 0$ and $\sigma(\dimv W) \leq 0$ for any subrepresentation $W$ of $V$.
Since the number of distinct $\dimv W$ is finite, the above linear system is also finite and hence defines a polyhedral cone.
We call the polyhedral cone the \emph{King cone} of a quiver representation $V$.
Interestingly, the King cone is related to the cone of feasible flows in network-flow problems.

Let us first consider the easiest case.
If $V$ is a quiver representation with $\dim V(i) = 1$ for all $i \in Q_0$, then
King's criterion characterizes the existence of a nonnegative flow $\varphi$ on the support quiver with the boundary condition $\partial \varphi = \sigma$.
To state it more precisely,
we introduce notation on flows.
Let $Q = (Q_0, Q_1)$ be a quiver (or a directed graph).
For a vertex subset $X \subseteq Q_0$,
let $\Out(X)$ denote the set of outgoing arcs from $X$, i.e., $\Out(X) \coloneqq \{ a = (i,j) : i \in X,\ j \in Q_0 \setminus X \}$.
Similarly, let $\In(X)$ denote the set of incoming arcs to $X$.
If $X = \{i\}$, we abbreviate $\Out(\{i\})$ and $\In(\{i\})$ as $\Out(i)$ and $\In(i)$, respectively.
If $\Out(X) = \emptyset$, then $X$ is called a \emph{lower set} of $Q$.
For a flow $\varphi \in \R^{Q_1}$ on $Q$,
its \emph{boundary} $\partial \varphi \in \R^{Q_0}$ is defined by $\partial \varphi(i) \coloneqq \sum_{a \in \Out(i)} \varphi(a) - \sum_{a' \in \In(i)} \varphi(a')$ for $i \in Q_0$.

Let us return to the $\sigma$-semistability of a representation $V$ with $\dim V(i) = 1$ for all $i \in Q_0$.
In this case, $V(a) \in \C$ for each arc $a \in Q_1$,
and a subrepresentation $W$ of $V$ can be identified with a vertex subset $X \subseteq Q_0$.
By the definition of a subrepresentation, if $i \in X$ and $W(a) \neq 0$ for $a = (i, j) \in Q_1$, then $j \in X$. 
This implies that $X$ is a lower set in the support quiver of $V$, namely, the subquiver of $Q$ whose arcs are $a \in Q_1$ with $V(a) \neq 0$.
Therefore, King's criterion is equivalent to the purely combinatorial condition that $\sigma(Q_0) = 0$ and $\sigma(X) \leq 0$ for each lower set $X$ of $Q$,
which characterizes the existence of a nonnegative flow $\varphi$ on the support quiver with the boundary condition $\partial \varphi = \sigma$
by Gale's theorem~\cite{Gale1957} (see, e.g., \cite[Theorem~9.2]{KorteVygen2018}).

By generalizing the above observation, we show that if $V$ is a \emph{rank-one} representation of $Q$, i.e., $V(a)$ is a rank-one matrix for each $a \in Q_1$,
then
\begin{itemize}
    \item King's criterion can be rephrased as a purely combinatorial condition with respect to the linear matroids arising from the rank-one matrices $V(a)$ of $Q$, and
    \item the rephrased condition above can be further viewed as the feasibility condition of a network flow-type problem called \emph{submodular flow}.
\end{itemize}
That is,
the King cone is representable as the feasibility of a certain instance of the submodular flow problem.
This enables us to decide the $\sigma$-semistability for rank-one representations in strongly polynomial time.

\begin{theorem}[informal version of \Cref{thm:submodular-flow,thm:rank-one:poly}]\label{thm:intro:rank-one}
Let $Q$ be an acyclic quiver, $V$ a rank-one representation of $Q$, and $\sigma$ a weight.
Then, $\sigma$ is in the King cone if and only if there is a feasible flow in the instance of submodular flow constructed from $V$, $\sigma$. Therefore, using standard submodular flow algorithms, we can decide the $\sigma$-semistability of rank-one representations in strongly polynomial time.
\end{theorem}

This theorem recovers the following well-known results when applied to the generalized Kronecker quiver and a star quiver.
\begin{itemize}
    \item A rank-one linear matrix $\sum_{k=1}^m x_k v_k f_k$ is (nc-)nonsingular (where $v_k$ is a column vector and $f_k$ a row vector) if and only if the linear matroids of $(f_k : k \in [m])$ and $(v_k : k \in [m])$ have a common base~\citep{Lovasz1989}.
    \item If each linear operator $B_i = f_i$ is of rank-one for $i \in [m]$ (where $f_i$ is a row vector), the BL polytope coincides with the base polytope of the linear matroid of $(f_i : i \in [m])$~\citep{Barthe1998}.
\end{itemize}

\subsection{Semistability of general quivers}\label{subsec:general-quiver}
Thus far, we have considered the $\sigma$-semistability of acyclic quivers.
As a complementary result, we show that the semistability of cyclic quivers under the $\GL(Q, \alpha)$-action can be efficiently reduced to \emph{noncommutative polynomial identity testing}.
In particular, we show that the polynomial can be represented by an \emph{algebraic branching program} (ABP).
This yields a deterministic algorithm for deciding the semistability of general quivers because noncommutative polynomial identity testing for ABP can be conducted in deterministic polynomial time~\citep{Raz2005}; see \Cref{sec:general-ss}.
Note that \citep{Burgisser2019} devised another deterministic algorithm for the problem with their framework of noncommutative optimization, which is built upon deep results in various areas of mathematics.
See also the discussion in related work.

\paragraph{Remark.}
After submitting the first version of this paper, an anonymous reviewer pointed out that this result for general quivers was sketched in \citet{Mulmuley2017}. See Theorem~10.8 and the last paragraph of Section~10.2 in \citet{Mulmuley2017}. 
At a very high level, our algorithm and Mulmuley's results follow a similar strategy, although Mulmuley's result requires several deep algebro-geometric backgrounds.
We believe that our proof is more explicit and elementary, and hence, is worthy to be presented here for completeness.

\subsection{Our techniques}
In this subsection, we outline our techniques.

\paragraph{$\sigma$-semistability.}
Our starting point is a reduction of general acyclic quivers to the generalized Kronecker quiver~\cite{Derksen2017,Huszar2021}.
We decompose the weight $\sigma = \sigma^+ - \sigma^-$, where $\sigma^+(i) \coloneqq \max\{\sigma(i), 0\}$ and $\sigma^-(i) \coloneqq \max\{-\sigma(i), 0\}$.
Let $Q_0^+$ and $Q_0^-$ be the sets of vertices $i$ such that $\sigma(i) > 0$ and $\sigma(i) < 0$, respectively.
\citet{Derksen2017} showed that the $\sigma$-semistability of a representation $V$ with the dimension vector $\alpha$ is equivalent to the nc-nonsingularity of a partitioned linear matrix
\begin{align}
    A: \bigoplus_{i \in Q_0^+} {\bigl(\C^{\alpha(i)}\bigr)}^{\oplus \sigma^+(i)} \to \bigoplus_{i \in Q_0^-} {\bigl(\C^{\alpha(i)}\bigr)}^{\oplus \sigma^-(i)},
\end{align}
where the $(s,p;t,q)$-block ($s \in Q_0^+$, $p \in [\sigma^+(s)]$, $t \in Q_0^-$, $q \in [\sigma^-(t)]$) of $A$ is given by a linear matrix
\[
    \sum_{P: \text{$s$--$t$ path}} x_{P,p,q} V(P).
\]
Here, $x_{P,p,q}$ is an indeterminate and $V(P)$ is the linear map corresponding to the path $P$, i.e., $V(P) \coloneqq V(a_k) \cdots V(a_1)$ for $P = (a_1, \dots, a_k)$ as a sequence of arcs.
However, because the number of indeterminates is exponential, applying the known nc-rank computation algorithms in a black box manner does not yield the desired time complexity. 

Inspired by the above reduction, we introduce the following scaling problem for quiver representations.
We define a scaling $V_{g,h}$ of the quiver representation $V$ by the block matrices $g = \bigoplus (g_t : t \in Q_0^-)$ and $h = \bigoplus (h_s : s \in Q_0^+)$ as
\begin{align}
    V_{g,h}(a) \coloneqq 
    \begin{cases}
        g_{ha} V(a) h_{ta}^{\dagger}  & \text{if $a \in \Out(Q_0^+) \cap \In(Q_0^-)$}, \\
        V(a) h_{ta}^{\dagger} & \text{if $a \in \In(Q_0^-) \setminus \Out(Q_0^+)$}, \\
        g_{ha} V(a)  & \text{if $a \in \Out(Q_0^+) \setminus \In(Q_0^-)$}, \\
        V(a) & \text{otherwise}.
    \end{cases}
\end{align}
Furthermore, let $(b^+, b^-) \in \bigoplus_{i \in Q_0^+}\Q^{\alpha(s)} \times \bigoplus_{t \in Q_0^-}\Q^{\alpha(t)}$ be vectors such that $b^+(s) = \frac{\sigma^+(s)}{N}\ones_{\alpha(s)}$ and $b^-(t) = \frac{\sigma^-(t)}{N}\ones_{\alpha(t)}$, where $N \coloneqq \sigma^+(\alpha) = \sigma^-(\alpha)$ and $\ones$ denotes the all-one vector.
We say that $V$ is \emph{approximately scalable} (to the marginals $(b^+, b^-)$) if for any $\eps > 0$, there exist block matrices $g$ and $h$ such that $V_{g,h}$ satisfies
\begin{align}
    \norm*{\sum_{s \in Q_0^+} \sum_{P: \text{$s$--$t$ path}} V_{g,h}(P) V_{g,h}(P)^\dagger - \Diag(b^-(t)) }_{\tr} &< \eps \qquad (t \in Q_0^-),\\
    \norm*{\sum_{t \in Q_0^-} \sum_{P: \text{$s$--$t$ path}} V_{g,h}(P)^\dagger V_{g,h}(P) - \Diag(b^+(s)) }_{\tr} &< \eps \qquad (s \in Q_0^+),
\end{align}
where the norm is the trace norm.\footnote{The choice of the trace norm is not important here; we can use any unitary invariant norm.}
This is an instance of \emph{operator scaling with specified marginals}~\citep{Franks2018,Burgisser2018a}.
The crucial observation is that even though there may exist exponentially many $s$--$t$ paths, the above matrix sum can be computed efficiently by exploiting the underlying quiver structure.
Therefore, we can use a simple iterative algorithm in \citep{Burgisser2018a} to check $V$ is approximately scalable for a fixed $\eps > 0$.
We can show that it runs in $O(\eps^{-2}\poly(|Q|, \alpha(Q_0), b))$ time, where $b$ is the bit complexity of $V$.
Furthermore, we show that it is sufficient to consider $\eps = O(1/N)$ to decide the $\sigma$-semistability of $V$.
This yields our algorithm for checking the $\sigma$-semistability of quiver representations.

\paragraph{Maximizers in King's criterion.}
We follow a similar approach to find a maximizer in King's criterion.
We use the above linear matrix of the reduction~\cite{Derksen2017} directly and show that the shrunk subspaces of the above linear matrix correspond to the maximizers of King's criterion.
Then, we show that the necessary operations in the recent shrunk subspace algorithm~\citep{Franks2023} can be performed efficiently, enabling us to find the inclusion-wise minimum maximizer of King's criterion efficiently.
Note that the correspondence between the shrunk subspaces and maximizers of King's criterion is shown in \cite{Huszar2021} using abstract algebra; we provide a more direct and elementary proof using submodularity.

\paragraph{HN-filtration.}
Our HN-filtration algorithm is based on \emph{principal partitions} of submodular systems~\cite{Fujishige2009}.
For a slope $\mu = \sigma/\tau$, we consider a parametric modular function 
\begin{align}
    f_\lambda(W) \coloneqq \lambda \tau(\dimv W) - \sigma(\dimv W)
\end{align}
on the subrepresentations $W$ of $V$, where $\lambda \in \R$ is a parameter.
Let $\caL(\lambda)$ denote the modular lattice of the minimizers of $f_\lambda$ and let $W^-(\lambda)$ and $W^+(\lambda)$ be the minimum and maximum minimizers of $f_\lambda$, respectively.
By the standard argument in principal partition, we show that $W^+(\lambda) \leq W^-(\lambda')$ for $\lambda > \lambda'$.
Furthermore, there must be a finite set of $\lambda$ such that $\caL(\lambda)$ consists of more than one element.
We call such a value of $\lambda$ a \emph{critical value}.
Let $\lambda_1 > \cdots > \lambda_k$ be the critical values.
Then, they induce the filtration
\begin{align}
    \{0\} = W^-(\lambda_1) < W^+(\lambda_1) = W^-(\lambda_2) < W^+(\lambda_2) =  \cdots = W^-(\lambda_k) < W^+(\lambda_k) = V.
\end{align}
We show that this coincides with the HN-filtration.
Each $W^-(\lambda)$ and $W^+(\lambda)$ can be found by our algorithm for maximizers of King's criterion for fixed $\lambda$. 
The possible candidates for critical values can be easily enumerated, enabling us to find the HN-filtration efficiently.

\paragraph{Strongly polynomial-time algorithm for rank-one representations.}
When $V$ is a rank-one representation,
each rank-one matrix $V(a)$ is representable as $v_a f_a$ for some nonzero vector $v_a \in V(ha)$ and nonzero dual vector $f_a \in V(ta)^*$.
Based on this representation, we assign each vertex $i \in Q_0$ to two linear matroids $\M_i^+$ and $\M_i^-$,
where the first is generated by $\{ f_a : f_a \in \Out(i) \}$
and the second by $\{ v_a : v_a \in \In(i) \}$.
Then, we can simulate a subrepresentation $W$ of $V$ as a lower set $X$ of the directed graph $D[V]$ constructed from $V$;
its vertex set is the (disjoint) union of the ground sets
$\{ f_a : f_a \in \Out(i) \}$
and $\{ v_a : v_a \in \In(i) \}$
of the matroids $\M_i^+$ and $\M_i^-$ for $i \in Q_0$;
its arc set represents $Q_1$ and the dependencies as
``if $W(ha) = W(tb)$ contains $v_a$ then $W(tb) \not\leq \ker f_b$; hence, $W(hb)$ must contain $v_b$ for $W$ to be a subrepresentation of $V$.''
This enables us to rephrase King's criterion as a combinatorial condition on the lower sets $X$ of $D[V]$ as
\begin{align}
    \sum_{i \in Q_0} \left( \sigma^+(i) \left(\dim V(i) - r_i^+(\{ f_a : f_a \in \Out(i) \} \setminus X) \right) - \sigma^-(i) r_i^-(\{ v_a : v_a \in \In(i) \} \cap X) \right) \leq 0,
\end{align}
where $r_i^+$ and $r_i^-$ denote the rank functions of $\M_i^+$ and $\M_i^-$,
respectively.
We further rephrase the above combinatorial condition as
the feasibility characterization of a certain instance of the submodular flow problem by Frank~\cite{Frank1984}.
Thus, we can check the $\sigma$-semistability of a rank-one representation $V$ by
checking the feasibility of the instance generated by $V$ of the submodular flow problem.

\paragraph{Semistability of general quivers.}
For the semistability of general quivers, we use an invariant polynomial characterization of the null-cone.
By the general theory of GIT, a representation $V$ is semistable if and only if there exists a $\GL(Q,\alpha)$-invariant homogeneous polynomial $p$ on the representation space $\Rep(Q, \alpha)$ such that $p(V) \neq 0$.
The Le Bruyn-Procesi theorem~\cite{Bruyn1990} stated that the ring of invariant polynomials is generated by polynomials in the form of
\begin{align}
    \tr[V(a_k)\cdots V(a_2)V(a_1)],
\end{align}
for a closed path\footnote{Here, a closed path means a sequence $(a_1, \dots, a_k)$ of arcs such that $ha_l = ta_{l+1}$ ($l \in [k]$), where $a_{k+1} \coloneqq a_1$. In graph theory, it is usually called a \emph{closed walk}. In this paper, we follow the standard terminologies in quiver representation.} $(a_1, a_2, \dots, a_k)$ in $Q$ with length $k \ge 1$.
Furthermore, closed paths with length $1 \le k \leq \alpha(Q_0)^2$ generate the invariant ring, where $\alpha = \dimv V$.

Therefore, we can decide the semistability by checking whether the above polynomial is nonzero at some vertex $i$ and closed path $C$.
The obstacle is that the number of closed paths can be exponential.
To this end, we consider another polynomial in \emph{noncommutative} indeterminate $x_a$ ($a \in Q_1$) defined as
\begin{align}
    \sum_{\text{$C$: closed path starting at $i$}} x^C\tr V(C)
\end{align}
for each vertex $i \in Q_0$, where $x^C = x_{a_k} \cdots x_{a_1}$ for $C = (a_1, \dots, a_k)$.
Then, $V$ is semistable if and only if this noncommutative polynomial is nonzero at some vertex $i$.
Note that the noncommutativity is essential to distinguish closed paths with the same arc sets.
For example, consider a quiver with a single vertex and two self-loops (say, $a$ and $b$).
Then, closed walks $C = abab \cdots ab$ and $C' = a^k b^k$ of length $2k$ have the same number of $a$ and $b$, but their trace $\tr V(C)$ and $\tr V(C')$ are different in general for $k \geq 2$.
So if we use commutative indeterminates, multiple closed walks correspond to a single monomial, and we cannot decide whether $\tr V(C) \neq 0$ for some $C$ or not by checking whether the polynomial is zero or not.

Yet, we need to show how to perform noncommutative polynomial identity testing for this polynomial in deterministic polynomial time.
Using the underlying quiver structure, we can show that this polynomial can written as an ABP of polynomial size.
Applying the algorithm of \citet{Raz2005}, we obtain our algorithm for the semistability of general quivers. 

We remark that this is the only place where a nontrivial algebraic result from the GIT machinery is needed.
The other algorithms and analysis can be understood with elementary linear algebra (assuming the known analysis of operator scaling algorithms, which involves some abstract algebra).

\subsection{Related work}
Existing studies on algorithms for quiver semistability have focused on \emph{bipartite} quivers~\citep{Chindris2021,Chindris2022a,Chindris2022b,Franks2023b}.
We remark that semistability in bipartite quivers is essentially operator scaling with a block structure; see \Cref{subsec:block-matrix}.
In bipartite quivers, the number of paths is equal to the number of arcs; hence, a weak running time was sufficient in the previous studies.
\cite{Huszar2021,Cheng2024} studied general acyclic quivers.
They first used the reduction of \citep{Derksen2017} to the generalized Kronecker quiver and applied the nc-rank algorithm~\cite{Ivanyos2018} in a black box manner.
Hence, their algorithm runs in time polynomial in the number of paths, which can be exponential.

Several polyhedral cones associated with quiver representations have been studied in the literature~\citep{Chindris2022c,Vergne2023}.
The \emph{moment cone} of a quiver $Q$ and a dimension vector $\alpha$ is the polyhedral cone generated by the highest weights of the representations of $Q$ with the dimension vector $\alpha$.
The membership of the moment cone can be decided in strongly polynomial time for bipartite quivers~\citep{Chindris2022c} and even general acyclic quivers~\citep{Vergne2023}.
Another polyhedral cone is the conic hull of weight $\sigma \in \Z^{Q_0}$ such that there exists a nonzero semi-invariant polynomial in $\Rep(Q, \alpha)$ with weight $\sigma$.
The membership problem of this cone is called the \emph{generic semistability problem}~\citep{Chindris2022c}.
By definition, $\sigma$ is in this cone if and only if there exists a \emph{generic} $\sigma$-semistable representation of $Q$, hence the name.
To the best of our knowledge, the generic semistability problem remains open for general acyclic quivers.

The semistability of quiver representations is a special case of semistability in the GIT.
In the most abstract setting, GIT studies group actions on algebraic varieties.
We say that a point in the variety is semistable if its orbit closure does not contain the origin.
\citet{Burgisser2019} proposed a framework of \emph{noncommutative optimization} to devise algorithms for GIT problems in the general setting.
Although noncommutative optimization is a broad and general framework, it does not always provide efficient algorithms for all GIT problems.
Currently, most of the known tractable problems originate from a family of operator scaling problems, which are also contained in the semistability of quiver representations.
Furthermore, it is built upon deep results in various areas of mathematics, such as algebraic geometry, Lie algebra, and representation theory, rendering it difficult for non-experts to understand.
Another conceptual contribution of this paper is identifying the semistability of quiver representations as a useful subclass of GIT.
The semistability of quiver representations is rich enough to capture various interesting problems in the literature while also supporting the design efficient algorithms using elementary techniques.

\subsection{Organization of this paper}
The rest of this paper is organized as follows.
\Cref{sec:prelim} introduces the necessary background and notation of quiver representations, operator scaling, and noncommutative rank.
\Cref{sec:ss} presents our algorithms for deciding the $\sigma$-semistability and finding maximizers of King's criterion.
\Cref{sec:HN} presents our HN-filtration algorithm and its application to the coarse DM-decomposition.
\Cref{sec:rank-one} investigates the King cone of rank-one representations and shows the reduction to submodular flow.
\Cref{sec:general-ss} describes our reduction from the semistability of general quivers to polynomial identity testing of noncommutative ABPs.

\section{Preliminaries}\label{sec:prelim}
We denote the set of nonnegative integers, rational, and real numbers by $\Z_+$, $\Q_+$, and $\R_+$, respectively.
We let $[m, n] \coloneqq \{m, m+1, \dots, n-1, n\}$ for $m,n \in \Z$ with $m \leq n$ and $[n] \coloneqq [1, n] = \{1, \dotsc, n\}$ for a positive integer $n$.
We denote the set of $m \times n$ complex matrices by $\Mat(m, n)$.
We simply denote $\Mat(n, n)$ by $\Mat(n)$.
The conjugate transpose of a matrix $A$ is denoted by $A^\dagger$.
The subgroup of the upper triangular matrices in $\GL(n)$ (i.e., the Borel subgroup) is denoted by $\B(n)$.
For two vector spaces $U$ and $V$, we mean by $U \le V$ that $U$ is a subspace of $V$.
Let $\spnn{S}$ denote the vector space spanned by a multiset $S$ of the vectors.
For a vector $b \in \C^n$, we denote by $\Diag(b)$ the $n \times n$ diagonal matrix such that the entries of $b$ are on the diagonal.

\subsection{Operator scaling, matrix space, and noncommutative rank}
A linear map $\Phi: \Mat(n) \to \Mat(m)$ is said to be \emph{completely positive}, or \emph{CP} for short, if $\Phi(X) = \sum_{\ell=1}^k A_\ell X A_\ell^\dagger$ for some $A_\ell \in \Mat(m, n)$.
These $A_\ell$ are called the \emph{Kraus operators} of $\Phi$.
The \emph{dual} map $\Phi^*: \Mat(m) \to \Mat(n)$ of $\Phi$ is defined by $\Phi^*(X) = \sum_{\ell=1}^k A_\ell^\dagger X A_\ell$.
For $(g, h) \in \GL(m) \times \GL(n)$, we define the scaling $\Phi_{g,h}$ of $\Phi$ by 
\begin{align}
    \Phi_{g,h}(X) \coloneqq g \Phi(h^\dagger X h) g^\dagger = \sum_{\ell=1}^k (g A_\ell h^\dagger) X (g A_\ell h^\dagger)^\dagger.
\end{align}
If $m = n$, the CP map is said to be \emph{square}.

Let $\Phi: \Mat(n) \to \Mat(n)$ be a square CP map.
Let $\ds(\Phi) \coloneqq \norm{\Phi(I) - I}_{\mathrm F}^2 + \norm{\Phi^*(I) - I}_{\mathrm F}^2$, where $\norm{\cdot}_{\mathrm F}$ denotes the Frobenius norm.
Then, $\Phi$ is said to be \emph{approximately scalable} if for any $\eps > 0$, there exists $(g, h) \in \GL(n) \times \GL(n)$ such that $\ds(\Phi_{g,h}) \leq \eps$.
The goal of the \emph{operator scaling} problem is to decide whether a given CP map is approximately scalable or not.

Operator scaling is closely related to the noncommutative rank (nc-rank) of linear matrices.
An $m \times n$ symbolic matrix $A$ is called a \emph{linear matrix} if $A = \sum_{\ell = 1}^k x_\ell A_\ell$ for indeterminates $x_\ell$ and matrices $A_\ell \in \Mat(m, n)$.
Sometimes, it is more convenient to see a linear matrix as a \emph{matrix space} spanned by $A_1, \dots, A_k$.
We denote the corresponding matrix space of a linear matrix $A$ by $\caA$.
For a subspace $U \leq \C^n$, let 
$
    \caA U \coloneqq \spnn{ \{A u : A \in \caA,\ u \in U\} },
    $
which is a subspace of $\C^m$.
The \emph{nc-rank} of a linear matrix $A$ (denoted by $\ncrank A$) is defined by
\begin{align}
    \ncrank A \coloneqq \min\{n + \dim\caA U - \dim U : U \leq \C^n\}.
\end{align}
A square linear matrix $A$ is said to be \emph{nc-nonsingular} if $\ncrank A = n$.
Informally speaking, $\ncrank A$ is the rank of $A$, where the indeterminates $x_i$ are pairwise noncommutative, i.e., $x_i x_j \neq x_j x_i$ for $i \neq j$.
See~\cite{Cohn1995,Fortin2004} for more details.
A pair $(L, R)$ of subspaces $L \leq \C^m$ and $R \leq \C^n$ is called an \emph{independent subspace} if $L \cap \caA R = \{0\}$.
Over the complex field, $(L, R)$ is independent if and only if $\tr(\Pi_L \Phi(\Pi_R)) = 0$, where $\Pi_L$ denotes the orthogonal projection matrix onto $L$.
Then, 
\begin{align}\label{eq:nc-rank-indep}
\ncrank A = m + n  - \max\{\dim L + \dim R : \text{$(L, R)$ an independent subspace}\}.
\end{align}
An independent subspace $(L, R)$ is said to be \emph{maximum} if $\dim L + \dim R$ is maximum.
In particular, a square linear matrix $A$ is nc-nonsingular if and only if $\dim L + \dim R \leq n$ for any independent subspace $(L, R)$.
Gurvits' theorem~\cite{Gurvits2004} states that a square CP map $\Phi$ with the Kraus operator $A_\ell$ ($\ell \in [k]$) is approximately scalable if and only if the linear matrix $\sum_{\ell=1}^k x_\ell A_\ell$ is nc-nonsingular.

\subsection{Finding minimum shrunk subspace}
\emph{Shrunk subspaces} are minimizers of the objective function in nc-rank:
\begin{align}
    f(U) = n + \dim\caA U - \dim U.
\end{align}
It is easy to see that $f$ is submodular, i.e., $f(U) + f(V) \geq f(U \cap V) + f(U + V)$ for any subspaces $U, V \leq \C^n$.
Therefore, the shrunk subspaces form a sublattice of $\C^n$, i.e., if $U$ and $V$ are shrunk subspaces, then $U \cap V$ and $U + V$ are also shrunk subspaces.
Hence, there exists a unique shrunk subspace with minimum dimension, which we call the \emph{minimum} shrunk subspace.
A shrunk subspace can be regarded as a certificate of the value of the nc-rank.
In the GIT perspective, shrunk subspaces correspond to one-parameter subgroups that bring the matrix tuple $(A_1, \dots, A_k)$ to the origin.
The minimum shrunk subspace is particularly important because it has a rational basis with polynomial bit complexity if $(A_1, \dots, A_k)$ has Gaussian integer entries~\cite{Ivanyos2018}.

Several algorithms can find the minimum shrunk subspace in polynomial time~\cite{Ivanyos2018,Franks2023}.
We use the algorithm in \cite{Franks2023}, which is based on a modified operator Sinkhorn iteration. We use it as a black box and 
consider only the square case for simplicity.
The details of the algorithm can be referred from~\cite{Franks2023}.

\begin{theorem}[\citet{Franks2023}]\label{thm:shrunk}
    Let $\Phi: \Mat(n) \to \Mat(n)$ be a square CP map whose Kraus operators have Gaussian integer entries.
    Let $A$ be the linear matrix corresponding to $\Phi$.
    Then, there exists a deterministic polynomial time algorithm that computes a basis of the minimum shrunk subspace of $A$. 
    Furthermore, this algorithm works even if $\Phi$ is given as an oracle that computes $\Phi(X)$ and $\Phi^*(X)$ for $X \in \Mat(n)$, along with an upper bound $b$ of the bit complexity of the Kraus operators.
    The time complexity is polynomial in $\poly(n, b)(\mathrm{EO} + O(n^3))$, where $\mathrm{EO}$ denotes the time complexity of a single oracle call. 
\end{theorem}

\subsection{Operator scaling with specified marginals}
\emph{Operator scaling with specified marginals} is a generalization of operator scaling.
Let $(b^+, b^-) \in \R^n_+ \times \R^m_+$ be a pair of nonincreasing nonnegative vectors, which we call the \emph{target marginals}.
We say that a CP map $\Phi: \Mat(n) \to \Mat(m)$ is \emph{approximately scalable to the target marginals $(b^+, b^-)$} if there exist nonsingular \emph{upper triangular} matrices $(g, h) \in \B(m) \times \B(n)$ such that 
\begin{align}
\norm{\Phi_{g,h}(I) - \Diag(b^-)}_{\tr} \leq \eps, \quad \norm{\Phi_{g,h}^*(I) - \Diag(b^+)}_{\tr} \leq \eps.
\end{align}
Such target marginals are said to be \emph{feasible}.
We define $\Delta b^+ \in \R_+^n$ as $\Delta b^+_j = b^+_j - b^+_{j+1}$ for $j \in [n]$, where we conventionally define $b^+_{n+1} \coloneqq 0$.
Similarly, we define $\Delta b^- \in \R_+^m$.
Let $F^+_j = \langle \be_1, \dots, \be_j \rangle$ be the standard flag of $\C^n$ for $j \in [n]$.
Similarly, we define $F^-_i$ for $i \in [m]$.
The following theorem characterizes the set of feasible marginals by a certain linear system.

\begin{theorem}[{\cite[Theorem~18]{Franks2018}}]\label{thm:opscaling-marginals}
    Let $\Phi: \Mat(n) \to \Mat(m)$ be a CP map and $(b^+, b^-) \in \R^n_+ \times \R^m_+$ a pair of nonincreasing nonnegative vectors.
    Then, $\Phi$ is approximately scalable to the marginals $(b^+, b^-)$ if and only if $\sum_{j=1}^n b^+_j = \sum_{i=1}^m b^-_i \eqqcolon B$ and 
    \begin{align}
        \sum_{i=1}^m \Delta b^-_i \dim(L \cap F^-_i) +  \sum_{j=1}^n \Delta b^+_j \dim(R \cap F^+_j) \leq B
    \end{align}
    for any independent subspace $(L,R)$ of $\Phi$.
\end{theorem}

Let $(b^+, b^-)$ be a feasible marginal with rational entries.
There is an efficient algorithm that finds a scaling of a given CP map $\Phi$ whose marginal is $\eps$-close to $(b^+, b^-)$~\citep{Franks2018,Burgisser2018a}.

\begin{theorem}[Theorem~1.13 in \citet{Burgisser2018a}, specialized for operator scaling]
    Let $\eps > 0$ be an accuracy parameter and $\Phi: \Mat(n) \to \Mat(m)$ a CP map with Gaussian integer Kraus operators.
    Let $(b^+, b^-) \in \Q^n \times \Q^m$ be target marginals such that $b^+_1 \geq \dots \geq b^+_n \geq 0$, $b^-_1 \geq \dots \geq b^-_m \geq 0$, and $\sum_{j=1}^n b^+_j = \sum_{i=1}^m b^-_i = 1$.
    Then, \Cref{alg:marginals} finds upper triangular $g, h$ such that $\norm{\Phi_{g,h}(I) - \Diag(b^-)}_{\tr} \leq \eps$ and $\norm{\Phi_{g,h}^*(I) - \Diag(b^+)}_{\tr} \leq \eps$ in 
    $
    T = O(\eps^{-2}(b + N \log (\ell N)))
    $
    iterations, where $b$ is the maximum bit length of the target marginals $(b^+, b^-)$, $N \coloneqq \max\{m, n\}$, and $\ell$ is the smallest positive integer such that $\ell(b^+, b^-)$ is an integer.
    Furthermore, each iteration can be executed in time $O(N^3)$.
\end{theorem}

\begin{algorithm}
    \caption{Operator Sinkhorn iteration for specified marginals~\citep{Franks2018,Burgisser2018a}\label{alg:marginals}}
\begin{algorithmic}[1]
    \For{$t=1, \dots, T$}
        \State \Comment{Left normalization}
        \State Compute the Cholesky decomposition $CC^\dagger = \Phi(I)$. Set $g = \Diag(b^-)^{1/2}C^{-1}$ and $\Phi \gets \Phi_{g,I}$.
        \State \Comment{Right normalization}
        \State Compute the Cholesky decomposition $CC^\dagger = \Phi^*(I)$. Set $h = \Diag(b^+)^{1/2}C^{-1}$ and $\Phi \gets \Phi_{I,h}$.
    \EndFor
\end{algorithmic}
\end{algorithm}

\subsection{Useful results for block matrices}\label{subsec:block-matrix}
We frequently use CP maps or linear matrices with block structures throughout the paper.
Hence we present some useful results here.
Let $V^+$ and $V^-$ be finite sets and let $\alpha(s)$ ($s \in V^+$) and $\alpha(t)$ ($t \in V^-$) be positive integers.
Let $n \coloneqq \sum_{s \in V^+} \alpha(s)$ and $m \coloneqq \sum_{t \in V^-} \alpha(t)$.
Consider a linear matrix $A$ of size $m \times n$ with the following block structure:
The $(s,t)$-block of $A$ is given by a linear matrix 
\begin{align}\label{eq:partitioned-matrix}
    \sum_{\ell = 1}^{k_{s,t}} x_{s,t,\ell} A_{s,t,\ell},
\end{align}
where $k_{s,t}$ is a nonnegative integer, $x_{s,t,\ell}$ is an indeterminate, and $A_{s,t,\ell}$ is an $\alpha(t) \times \alpha(s)$ matrix.
We remark that the indeterminate $x_{s,t,\ell}$ appears only in the $(s,t)$-block of $A$.
Let $\caA_{s,t}$ denote the matrix subspace spanned by $A_{s,t,\ell}$.

The following is a useful lemma that shows the existence of a shrunk subspace respecting the block structure.

\begin{lemma}\label{lem:partitoned-shrunk}
    For a partitioned linear matrix in the form \eqref{eq:partitioned-matrix},
    every shrunk subspace $U$ is in the form $U = \bigoplus_{s \in V^+} U_s$, where $U_s \leq \C^{\alpha(s)}$.
\end{lemma}
\begin{proof}
    For any $U \leq \C^n$, let $U_s$ denote the projection of $U$ to $\C^{\alpha(s)}$ for $s \in V^+$.
    Then, we have
    \[
        \caA U = \bigoplus_{t \in V^-} \sum_{s \in V^+} \caA_{s,t}U_s.
    \]
    Hence, replacing $U$ with $U' \coloneqq \bigoplus_{s \in V^+} U_s$ does not change $\caA U$.
    If $U \neq U'$, then $\dim U < \dim U'$.
    Hence, we have $\dim\caA U - \dim U > \dim\caA U' - \dim U'$,
    which implies that $U$ is not a shrunk subspace.
\end{proof}

The following is deduced from the modular lattice structure of the optimal shrunk subspaces.

\begin{lemma}\label{lem:partitoned-shrunk-rep}
    Let $A$ be a partitioned linear matrix as in \eqref{eq:partitioned-matrix} and $U = \bigoplus_{s \in V^+} U_s$ the minimal shrunk subspace of $A$.
    Suppose that there exist $s, s' \in V^+$ such that $\caA_{s,t} = \caA_{s',t}$ for all $t \in V^-$.
    Then, $U_s = U_{s'}$.
\end{lemma}

Let us consider the corresponding CP map $\Phi: \bigoplus_{s \in V^+}\Mat(n_s) \to \bigoplus_{t \in V^-} \Mat(n_t)$ that maps $X = \bigoplus_{s \in V^+} X_s$ to $\Phi(X) = \bigoplus_{t \in V^-}\Phi(X)$, where
\begin{align}\label{eq:block-CP}
    \Phi(X)_t = \sum_{s \in V^+}\sum_{\ell = 1}^{k_{s,t}} A_{s,t,\ell} X_s A_{s,t,\ell}^\dagger.
\end{align}
Here is a version of \Cref{thm:opscaling-marginals} for CP maps with a block structure.
We say that a matrix $h \in \bigoplus_{s \in V^+} \Mat(n_s)$ is \emph{block-wise upper triangular} if $h = \bigoplus_{s \in V^+} h_s$ where $h_s$ is upper triangular for all $s \in V^+$.

\begin{lemma}[{\cite[Proposition~61]{Franks2018}}]\label{lem:opscaling-marginals-block}
    Let $\Phi$ be a CP map in the form \eqref{eq:partitioned-matrix}.
    Let $(b^+, b^-) = (\bigoplus_{s \in V^+} b^+(s), \bigoplus_{t \in V^-} b^-(s)) \in \R^n_+ \times \R^m_+$ be a pair of nonnegative vectors such that $b^+(s) \in \R_+^{\alpha(s)}$ and $b^-(t) \in \R_+^{\alpha(t)}$ are nonincreasing for $s \in V^+$ and $t \in V^-$.
    Then, $\Phi$ is approximately scalable to the marginals $(b^+, b^-)$ if and only if $\sum_{s \in V^+}\sum_{j=1}^{\alpha(s)} b^+(s)_j = \sum_{t \in V^-} \sum_{i=1}^{\alpha(t)} b^-(t)_i \eqqcolon B$ and
    \begin{align}
        \sum_{t \in V^-}\sum_{i=1}^{\alpha(t)} \Delta b^-(t)_i \dim(L_t \cap F^-(t)_i) +  \sum_{s \in V^+} \sum_{j=1}^{\alpha(s)} \Delta b^+(s)_j \dim(R_s \cap F^+(s)_j) \leq B
    \end{align}
    for any independent subspace $(L,R) = (\bigoplus_{t \in V^-} L_t, \bigoplus_{s \in V^+} R_s)$.
    Furthermore, scaling matrices can be taken to be block-wise upper triangular.
\end{lemma}

The target marginal often has the same block structure.
We say that the target marginal $(b^+, b^-)$ respects the block structure if $b^+$ and $b^-$ are constant within each block.

\begin{lemma}\label{lem:block-marginal}
    Let  $\Phi$ be a CP map with a block structure as in \eqref{eq:block-CP} and $(b^+, b^-)$ a target marginal respecting the same block structure.
    Let $\eps > 0$ be an accuracy parameter.
    Then, $\Phi$ can be scaled by block-wise upper triangular matrices $g, h$ such that $\norm{\Phi_{g,h}(I) - \Diag(b^-)}_{\tr} \leq \eps$ and $\norm{\Phi_{g,h}^*(I) - \Diag(b^+)}_{\tr} \leq \eps$ if and only if the same is possible with block nonsingular matrices $g, h$.
\end{lemma}
\begin{proof}
    By abusing the notation, we denote the common value of the entries of the $s$th block of $b^+$ by $b^+(s)$; the same is used for $b^-(t)$ as well.
    First, observe that for block unitary matrices $g, h$ 
    \begin{align}
        \Phi_{g,h}(I)_t - \Diag(b^-)_t &= g_t \Phi(h^\dagger h)_t g_t^\dagger - b^-(t)I_{\alpha(t)}
        = g_t(\Phi(I)_t - \Diag(b^-)_t)g_t^\dagger, \\ 
        \Phi^*_{g,h}(I)_s - \Diag(b^+)_s &= h_s \Phi^*(g^\dagger g)_s h_s^\dagger - b^+(s)I_{\alpha(s)} = h_s(\Phi^*(I)_s - \Diag(b^+)_s)h_s^\dagger
    \end{align}
    for $s \in V^+$ and $t \in V^-$.
    Since the trace norm is unitary invariant, the approximate scalability condition (i.e., $\norm{\Phi_{g,h}(I) - \Diag(b^-)}_{\tr} \leq \eps$ and $\norm{\Phi_{g,h}^*(I) - \Diag(b^+)}_{\tr} \leq \eps$) does not change under the (left) multiplication of block unitary matrices to $g, h$.
    Therefore, if there exist nonsingular block matrices $g, h$ that satisfy the approximate scalability condition, the R-factors of the QR decompositions of $g, h$ satisfy the same condition.
    Therefore, one can take block-wise upper triangular $g, h$.
    The other direction is trivial.
\end{proof}

\section{Reduction to nc-rank and algorithms for semistability}\label{sec:ss}
In this section, we present our algorithms for deciding $\sigma$-semistability and finding maximizers of King's criterion.

\subsection{Reduction from semistability to nc-rank}
Here, we recall the reduction of the $\sigma$-semistability of an acyclic quiver to nc-nonsingularity testing~\citep{Derksen2017}.

Let $Q$ be an acyclic quiver and $V$ a representation of $Q$ with the dimension vector $\alpha$.
Let $\sigma \in \Z^{Q_0}$ be a weight.\footnote{We do not assume $\sigma(\alpha) = 0$ here because the reduction does not need it. This is useful for finding a maximizer in King's criterion.}
Let $Q_0^+$ and $Q_0^-$ be the sets of vertices $i$ such that $\sigma(i) > 0$ and $\sigma(i) < 0$, respectively.
Let $\sigma^+(i) \coloneqq \max\{\sigma(i), 0\}$ and $\sigma^-(i) \coloneqq \max\{-\sigma(i), 0\}$ for each $i \in Q_0$.
Note that $\sigma = \sigma^+ - \sigma^-$.
Let $N \coloneqq \sigma^+(\alpha) = \sigma^-(\alpha)$.
We define an $N \times N$ partitioned linear matrix $A$ as follows.
As a linear map, 
\begin{align}\label{eq:A-Derksen2017}
    A: \bigoplus_{s \in Q_0^+} {\bigl(\C^{\alpha(s)}\bigr)}^{\oplus \sigma^+(s)} \to \bigoplus_{t \in Q_0^-} {\bigl(\C^{\alpha(t)}\bigr)}^{\oplus \sigma^-(t)}.
\end{align}
The $(s,p;t,q)$-block ($s \in Q_0^+$, $p \in [\sigma^+(s)]$, $t \in Q_0^-$, $q \in [\sigma^-(t)]$) of $A$ is given by a linear matrix
\[
    \sum_{P: \text{$s$--$t$ path}} x_{P,p,q} V(P),
\]
where $x_{P,p,q}$ is an indeterminate and $V(P)$ is the linear map corresponding to the path $P$, i.e., $V(P) \coloneqq V(a_k) \cdots V(a_1)$ for $P = (a_1, \dots, a_k)$ as a sequence of arcs.
The number of indeterminates in $A$ is equal to 
\begin{align}
    \sum_{s \in Q_0^+} \sum_{t \in Q_0^-} \sigma^+(s)\sigma^-(t) m(s,t),
\end{align}
where $m(s,t)$ denotes the number of $s$--$t$ paths in $Q$.
Thus, the number of indeterminates may be exponential in general.

The following lemma connects the nc-rank of $A$ with King's criterion, which is shown in \citep[Theorem~3.3]{Huszar2021} using abstract algebra.
We present an elementary proof for completeness. 

\begin{lemma}\label{lem:shrunk-King}
    The minimal maximizer of $\sigma(\dimv W)$ for the subrepresentations $W$ of $V$ corresponds to the minimal shrunk subspace of $A$.
\end{lemma}
\begin{proof}
    Let $W$ be a subrepresentation of $V$. 
    Define a subspace 
    \begin{align}
        U \coloneqq \bigoplus_{i \in Q_0^+} W(i)^{\oplus \sigma^+(i)} \leq \bigoplus_{i \in Q_0^+} (\C^{\alpha(i)})^{\oplus \sigma^+(i)}.
    \end{align}
    Inductively, we can show that $\caA U \leq \bigoplus_{i \in Q_0^-} W(i)^{\oplus\sigma^-(i)}$ since $W$ is a subrepresentation.
    Hence, 
    \begin{align}
        \dim U - \dim \caA U \geq \sum_{i \in Q_0^+} \sigma^+(i)\dim W(i) - \sum_{i \in Q_0^-} \sigma^-(i)\dim W(i) = \sigma(\dimv W).
    \end{align}

    On the other hand, let $U \leq \C^N$ be the minimal subspace that maximizes $\dim U - \dim \caA U$.
    By the block structure of $A$ and \Cref{lem:partitoned-shrunk,lem:partitoned-shrunk-rep}, we can decompose $U$ as $U = \bigoplus_{i \in Q_0^+} U_i^{\oplus\sigma^+(i)}$ for some $U_i \leq \C^{\alpha(i)}$.
    Define a subrepresentation $W$ recursively in a topological order of $Q$ as follows.
    If $i \in Q_0^+$, define $W(i) = U_i$. Otherwise, define $W(i) = \sum_{a \in \In(i)} V(a)W(ta)$.
    We show that $W$ is indeed a subrepresentation of $V$. 
    The only nontrivial case is when there exists a vertex $i \in Q_0^+$ such that $\sum_{a \in \In(i)} V(a) W(ta) \not\leq U_i$.
    Suppose that this happens.
    Denote $\sum_{a \in \In(i)} V(a) W(ta)$ by $W^-(i)$.
    Then, let $U'_i = U_i + W^-(i)$ and $U'_j = U_j$ for the rest.
    This gives a subspace $U'$ that has a strictly larger value of $\dim U' - \dim \caA U'$ because $\caA U' = \caA U$, contradicting the assumption of $U$.
    Thus, $W$ is a subrepresentation of $V$.
    By construction, we have $\caA U = \bigoplus_{i \in Q_0^-} W(i)^{\oplus\sigma^-(i)}$.
    Therefore, 
    \begin{align}
        \dim U - \dim \caA U = \sum_{i \in Q_0^+} \sigma^+(i)\dim W(i) - \sum_{i \in Q_0^-} \sigma^-(i)\dim W(i) = \sigma(\dimv W)
    \end{align}
    holds as required.
\end{proof}

By the lemma, one can check the $\sigma$-semistability of a quiver representation by checking whether the corresponding linear matrix is nc-nonsingular or not.
However, the naive reduction does not give a polynomial time algorithm, as the number of indeterminates in the linear matrix may be exponential.

\subsection{Scaling algorithm for $\sigma$-semistability}
We present a scaling algorithm for deciding the $\sigma$-semistability.
The idea is to reduce the problem to operator scaling with specified marginals. 

Let us define a CP map $\Phi_V$ corresponding to a quiver representation $V$.
As a linear map,
\begin{align}
    \Phi_V : \bigoplus_{s \in Q_0^+} \Mat(\alpha(s)) \to \bigoplus_{t \in Q_0^-} \Mat(\alpha(t)).
\end{align}
Let $X = \bigoplus (X_s : s \in Q_0^+)$ be an input block matrix.
Then,
\begin{align}
    (\Phi_V(X))_t \coloneqq  \sum_{s \in Q_0^+}\sum_{P: \text{$s$--$t$ path}} V(P) X_{s} V(P)^\dagger
\end{align}
for $t \in Q_0^-$.
The dual map 
\begin{align}
    \Phi_V^*: \bigoplus_{t \in Q_0^-} \Mat(\alpha(t)) \to \bigoplus_{s \in Q_0^+} \Mat(\alpha(s))
\end{align}
is given by
\begin{align}
    (\Phi_V^*(Y))_s \coloneqq  \sum_{t \in Q_0^-}\sum_{P: \text{$s$--$t$ path}} V(P)^\dagger Y_{t} V(P)
\end{align}
for $s \in Q_0^+$.
Analogous to the scaling of CP maps, we define a scaling $V_{g,h}$ of the quiver representation $V$ by block matrices $g = (g_t : t \in Q_0^-)$ and $h = (h_s : s \in Q_0^+)$ as
\begin{align}
    V_{g,h}(a) \coloneqq 
    \begin{cases}
        g_{ha} V(a) h_{ta}^{\dagger}  & \text{if $a \in \Out(Q_0^+) \cap \In(Q_0^-)$}, \\
        V(a) h_{ta}^{\dagger} & \text{if $a \in \In(Q_0^-) \setminus \Out(Q_0^+)$}, \\
        g_{ha} V(a)  & \text{if $a \in \Out(Q_0^+) \setminus \In(Q_0^-)$}, \\
        V(a) & \text{otherwise}.
    \end{cases}
\end{align}
Here is the key lemma that relates the $\sigma$-semistability of a quiver representation to the feasibility of a specific marginal of $\Phi_V$.

\begin{lemma}
    Let $V$ be a representation of an acyclic quiver $Q$ with the dimension vector $\alpha$ and $\sigma$ a weight with $\sigma(\alpha) = 0$.
    Let $(b^+, b^-)$ be the target marginals such that $b^+(s) = \frac{\sigma^+(s)}{N}\ones_{\alpha(s)}$ and $b^-(t) = \frac{\sigma^-(t)}{N}\ones_{\alpha(t)}$, where $N \coloneqq \sigma^+(\alpha) = \sigma^-(\alpha)$.
    Then, $\Phi_V$ is approximately scalable to the marginals $(b^+, b^-)$ if and only if $V$ is $\sigma$-semistable.
\end{lemma}
\begin{proof}
    We show that the conditions for scalability in \Cref{lem:opscaling-marginals-block} is equivalent to $A$ being nc-nonsingular, which in turn is equivalent to $V$ being $\sigma$-semistable.
    By construction, 
    \begin{align}
        \Delta b^+(s)_i = 
        \begin{cases}
            \frac{\sigma^+(s)}{N} & \text{if $i = \alpha(s)$}, \\
            0 & \text{otherwise},
        \end{cases}
        \qquad
        \Delta b^-(t)_j = 
        \begin{cases}
            \frac{\sigma^-(t)}{N} & \text{if $j = \alpha(t)$}, \\
            0 & \text{otherwise}.
        \end{cases}
    \end{align}
    Thus, the condition in \Cref{lem:opscaling-marginals-block} reduces to
    \begin{align}\label{eq:ss-indep}
        \sum_{t \in Q_0^-} \sigma^-(t) \dim L_t + \sum_{s \in Q_0^+} \sigma^+(s) \dim R_s \leq N
    \end{align}
    for any independent subspace $(L, R)$.
    This shows the nc-nonsingularity of $A$ by the formula of nc-rank in terms of independent subspaces~\eqref{eq:nc-rank-indep}.
\end{proof}

To check the feasibility of the marginal, one can use the scaling algorithm for operator scaling with the specified marginals (\Cref{alg:marginals}).
This yields \cref{alg:scaling-ss}.

\begin{algorithm}
    \caption{Scaling algorithm for $\sigma$-semistability\label{alg:scaling-ss}}
    \begin{algorithmic}[1]
        \Require{a representation $V$ of an acyclic quiver $Q$ and a weight $\sigma$}
        \State Let $b^+(s) \coloneqq \frac{\sigma^+(s)}{N} \ones_{\alpha(s)}$ and $b^-(t) \coloneqq \frac{\sigma^-(t)}{N} \ones_{\alpha(t)}$, where $N = \max\{\sigma^+(\alpha), \sigma^-(\alpha)\}$.
        \State Set $\eps = \frac{1}{6N}$ and $T \coloneqq O(\eps^{-2}(b + d \log (Nd)))$, where $d = \max\{\alpha(Q_0^+), \alpha(Q_0^-)\}$, and $b$ is the maximum bit length of $\sigma$.
        \For{$t=1, \dots, T$}
        \State \Comment{Left normalization}
        \State If $\norm{\Phi_V(I) - \Diag(b^-)}_{\tr} \leq \eps$ then \Return \texttt{Yes}.
        \State Compute the Cholesky decomposition $CC^\dagger = \Phi_V(I)$. Set $g = \Diag(b^-)^{1/2}C^{-1}$ and update $V \gets V_{g, I}$.
        \State \Comment{Right normalization}
        \State If $\norm{\Phi^*_V(I) - \Diag(b^+)}_{\tr} \leq \eps$ then \Return \texttt{Yes}.
        \State Compute the Cholesky decomposition $CC^\dagger = \Phi_V^*(I)$. Set $h = \Diag(b^+)^{1/2}C^{-1}$ and update $V \gets V_{I, h}$.
        \EndFor
        \State \Return \texttt{No}.
    \end{algorithmic}
\end{algorithm}

To run the algorithm, we first need to show how to compute the value of $\Phi_V$ in polynomial time. 
Note that the naive computation of $\Phi_V$ requires exponential time, as the number of terms in the sum is exponential.
However, we can compute $\Phi_V$ in polynomial time using the underlying quiver structure.
We show an algorithm in \Cref{alg:CP}.
The algorithm for the dual map is similar and therefore omitted.

\begin{algorithm}
    \caption{Algorithm for computing $\Phi_V(X)$.\label{alg:CP}}
    \begin{algorithmic}[1]
        \Require{a representation $V$ of an acyclic quiver $Q$ and a block matrix $X = (X_s : s \in Q_0^+)$}
        \State Let $X_i \coloneqq O$ for $i \in Q_0 \setminus Q_0^+$.
        \For{$a \in Q_1$ in a topological order} 
            \State $X_{ha} \gets X_{ha} + V(a) X_{ta} V(a)^\dagger$.
        \EndFor
        \State \Return $(X_i : i \in Q^-_0)$.
    \end{algorithmic}
\end{algorithm}

Next, we show an upper bound of the accuracy parameter $\eps$ that is sufficient to decide the $\sigma$-semistability.

\begin{lemma}\label{lem:scaling-ss-eps}
    Let $(b^+, b^-)$ be as above.
    Let 
    $
    0 < \eps \leq \frac{1}{6N}.
    $
    If there exist $g, h$ such that $\norm{(\Phi_V)_{g,h}(I) - \Diag(b^-)}_{\tr} \leq \eps$ and $\norm{(\Phi_V)_{g,h}^*(I) - \Diag(b^+)}_{\tr} \leq \eps$, then $V$ is $\sigma$-semistable.
\end{lemma}
\begin{proof}
    Let $\tilde \Phi_V$ be the scaled CP map of $\Phi_V$ by $g, h$.
    Let $\tilde b^+, \tilde b^-$ be the spectra of $\tilde \Phi_V(I), \tilde \Phi_V^*(I)$, respectively.
    First, observe that $(\tilde b^+, \tilde b^-)$ is also a feasible marginal.
    To see this, note that
    \begin{align}
        \tilde\Phi_V(I)_t - \Diag(b^-(t)) = U_t \left(\Diag(\tilde b^-) - \frac{\sigma^-(t)}{N} I_{\alpha(t)}\right) U_t^\dagger, 
    \end{align}
    where $U_t$ is the unitary matrix that diagonalizes $\tilde \Phi_V(I)$.
    Similarly, 
    \begin{align}
        \tilde\Phi_V^*(I)_s - \Diag(b^+(s)) = U_s\left(\Diag(\tilde b^+) - \frac{\sigma^+(s)}{N} I_{\alpha(s)}\right) U_s^\dagger. 
    \end{align}
    Hence, $g' = {\bigl(\bigoplus_{t \in Q_0^-} U_t\bigr)}^{-1} g$ and $h' = {\bigl( \bigoplus_{s \in Q_0^+} U_s\bigr)}^{-1} h$ satisfy the same assumption.
    By \Cref{lem:block-marginal}, there is no difference between the approximate scalability with upper triangular matrices and nonsingular matrices because $(b^+, b^-)$ is constant within each block.
    Therefore, $(\tilde b^+, \tilde b^-)$ is a feasible marginal under block triangular scaling.
    Furthermore, 
    \begin{align}
    \norm[\big]{\tilde b^+ - b^+}_1 &\leq \norm[\big]{\tilde\Phi_V(I) - \Diag(b^-)}_{\tr}   \leq \eps, \\
    \norm[\big]{\tilde b^- - b^-}_1 &\leq \norm[\big]{\tilde\Phi^*_V(I) - \Diag(b^+)}_{\tr} \leq \eps.
    \end{align}

    Since $(\tilde b^+, \tilde b^-)$ is feasible, from \Cref{lem:opscaling-marginals-block}, we have
    \begin{align}
        \sum_{t \in Q_0^-}\sum_{i=1}^{\alpha(t)} \Delta \tilde b^-(t)_i \dim(L_t \cap F^-(t)_i) +  \sum_{s \in Q_0^+} \sum_{j=1}^{\alpha(s)} \Delta \tilde b^+(s)_j \dim(R_s \cap F^+(s)_j) \leq \tilde B,
    \end{align}
    where $\tilde B$ denotes the sum of the entries of $\tilde b^+$.
    Let $\Delta \dim(L_t \cap F^-(t)_i) \coloneqq \dim(L_t \cap F^-(t)_i) - \dim(L_t \cap F^-(t)_{i-1})$, where we conventionally define $F^-(t)_0 = \{0\}$.
    Similarly, define $\Delta \dim(R_s \cap F^+(s)_j)$.
    Rewriting the LHS, we have
    \begin{align}
        \sum_{t \in Q_0^-}\sum_{i=1}^{\alpha(t)} \tilde b^-(t)_i \Delta \dim(L_t \cap F^-(t)_i)
         +  \sum_{s \in Q_0^+} \sum_{j=1}^{\alpha(s)} \tilde b^+(s)_j \Delta \dim(R_s \cap F^+(s)_j) \leq \tilde B.
    \end{align}
    Observe that
    \begin{align}
        \sum_{t \in Q_0^-}\sum_{i=1}^{\alpha(t)} b^-(t)_i \Delta \dim(L_t \cap F^-(t)_i)
         &\leq \sum_{t \in Q_0^-}\sum_{i=1}^{\alpha(t)} (\abs{b^-(t)_i - \tilde b^-(t)_i} + \tilde b^-(t)_i) \Delta \dim(L_t \cap F^-(t)_i) \\
         &\leq \norm[\big]{b^- - \tilde b^-}_1 + \sum_{t \in Q_0^-}\sum_{i=1}^{\alpha(t)}\tilde b^-(t)_i \Delta \dim(L_t \cap F^-(t)_i) \\
         &\leq \eps + \sum_{t \in Q_0^-}\sum_{i=1}^{\alpha(t)}\tilde b^-(t)_i \Delta \dim(L_t \cap F^-(t)_i). 
    \end{align}
    where the second inequality follows from $\Delta \dim(L_t \cap F^+(t)_i) \leq 1$.
    Similarly,
    \begin{align}
        \sum_{s \in Q_0^+}\sum_{j=1}^{\alpha(s)} b^+(s)_j \Delta \dim(R_s \cap F^+(s)_j)
        \leq \eps + \sum_{s \in Q_0^+}\sum_{j=1}^{\alpha(s)} \tilde b^+(s)_j \Delta \dim(R_s \cap F^+(s)_j).
    \end{align}
    Furthermore,
    \begin{align}
        \tilde B = \sum_{s \in Q_0^+} \sum_{j=1}^{\alpha(s)} \tilde b^+(s)_j 
                 \leq \sum_{s \in Q_0^+} \sum_{j=1}^{\alpha(s)} (\abs{b^+(s)_j - \tilde b^+(s)_j} + b^+(s)_j) 
                 \leq \eps + 1.
    \end{align}
    Combining these inequalities with the assumption that $0 < \eps \leq 1/6N$, we obtain
    \begin{align}
        \sum_{t \in Q_0^-}\sum_{i=1}^{\alpha(t)} \Delta b^-(t)_i \dim(L_t \cap F^-(t)_i) +  \sum_{s \in Q_0^+} \sum_{j=1}^{\alpha(s)} \Delta b^+(s)_j \dim(R_s \cap F^+(s)_j) \leq 1 + 3\eps \leq 1 + \frac{1}{2N}.
    \end{align}
    Therefore, we have
    \begin{align}
        \sum_{t \in Q_0^-} \sigma^-(t) \dim L_t +  \sum_{s \in Q_0^+} \sigma^+(s) \dim R_s \leq N + \frac{1}{2}.
    \end{align}
    By the integrality of the LHS and $N$, we can get rid of the $\frac{1}{2}$ term.
    Therefore, we obtain \eqref{eq:ss-indep} and show that $V$ is $\sigma$-semistable.
\end{proof}

\begin{theorem}\label{thm:ss}
    Let $V$ be a representation of an acyclic quiver $Q$ with Gaussian integer entries and $\sigma$ be a weight with $\sigma(\alpha) = 0$.
    \Cref{alg:scaling-ss} correctly decides the $\sigma$-semistability of $V$ in $O(N^2(b + d \log (Nd)))$ iterations, where $N = \sigma^+(\alpha) = \sigma^-(\alpha)$, $d = \max\{\alpha(Q_0^+), \alpha(Q_0^-)\}$, and $b$ is the maximum bit length of $\sigma$.
    Each iteration can be executed in $O(|Q_1|\alpha_{\max}^3 + \sum_{s \in Q_0^+} \alpha(s)^3 + \sum_{t \in Q_0^-} \alpha(t)^3)$ time, where $\alpha_{\max} = \max_{i \in Q_0} \alpha(i)$.
\end{theorem}
\begin{proof}
    If $V$ is $\sigma$-semistable, then $\Phi_V$ is approximately scalable to the marginals $(b^+, b^-)$. 
    By \Cref{thm:opscaling-marginals}, the algorithm must find such a scaling within the stated iterations.
    Consequently, the algorithm outputs \texttt{Yes}.
    If $V$ is not $\sigma$-semistable, then there is no scaling $g, h$ such that $\norm{(\Phi_V)_{g,h}(I) - \Diag(b^-)}_{\tr} \leq \eps$ and $\norm{(\Phi_V)_{g,h}^*(I) - \Diag(b^+)}_{\tr} \leq \eps$ for $\eps = 1/6N$ by \Cref{lem:scaling-ss-eps}.
    Thus, the algorithm outputs \texttt{No}.
    This proves the correctness of the algorithm.

    The number of iterations is immediate from the algorithm.
    The time complexity of each iteration is dominated by the computations of $\Phi_V(I)$ and $\Phi_V^*(I)$ and that of the block Cholesky decomposition.
    The former takes $O(|Q_1| \alpha_{\max}^3)$ time and the latter takes $O(\sum_{t \in Q_0^-} \alpha(t)^3 + \sum_{s \in Q_0^+} \alpha(s)^3)$ time.
\end{proof}

\subsection{Finding the extreme maximizer in King's criterion}
In this subsection, we extend the result from the previous section to find the extreme maximizer in King's criterion.
The idea is to use the shrunk subspace algorithm for the linear matrix \eqref{eq:A-Derksen2017}.

Let 
\begin{align}
    \Phi_V^\sigma: \bigoplus_{i \in Q_0^+} {\Mat(\alpha(i))}^{\sigma^+(i)} \to \bigoplus_{i \in Q_0^-} {\Mat(\alpha(i))}^{\sigma^-(i)}
\end{align}
be a CP map that maps $X = \bigoplus(X_{s,p} : s \in Q_0^+, p \in [\sigma^+(s)])$ to 
\begin{align}
    (\Phi_V^\sigma(X))_{t,q} = \sum_{s \in Q_0^+}\sum_{p \in [\sigma^+(s)]} \sum_{P: \text{$s$--$t$ path}} V(P) X_{s,p} V(P)^\dagger
\end{align}
for $t \in Q_0^-$ and $q \in [\sigma^-(t)]$.
Let  $J^+_\sigma = \bigoplus_{s \in Q_0^+} \sigma^+(s) I_{\alpha(s)}$ and $J^-_\sigma = \bigoplus_{t \in Q_0^-} \sigma^-(t) I_{\alpha(t)}$.
Note that
\begin{align}
    (\Phi_V^\sigma(I))_{t,q} = \sum_{s \in Q_0^+} \sigma^+(s) \sum_{P: \text{$s$--$t$ path}} V(P) V(P)^\dagger = \Phi_V(J^+_\sigma)_t, \\
    ((\Phi_V^\sigma)^*(I))_{s,p} = \sum_{t \in Q_0^-} \sigma^-(t) \sum_{P: \text{$s$--$t$ path}} V(P)^\dagger V(P) = \Phi_V^*(J^-_\sigma)_s.
\end{align}
Thus, one can compute $\Phi_V^\sigma(I)$ and $(\Phi_V^\sigma)^*(I)$ in $\poly(|Q|, |\alpha|, |\sigma|)$ time using \Cref{alg:CP}.
By \Cref{lem:shrunk-King}, the minimum shrunk subspace of $\Phi_V^\sigma$ corresponds to the minimum maximizer in King's criterion.
To find the minimum shrunk subspace of $\Phi_V^\sigma$, one can simply use \Cref{thm:shrunk}, which runs in $\poly(|Q|, |\alpha|, |\sigma|, b)$ time, because one can compute $\Phi_V^\sigma(I)$ and $(\Phi_V^\sigma)^*(I)$ in $\poly(|Q|, |\alpha|, |\sigma|, b)$ time.

The \emph{maximum} maximizer can also be found by considering $\Phi^*$ instead of $\Phi$.
To see this, observe that the maximum shrunk subspace corresponds to the maximum independent subspace $(L, R)$ such that $\dim L$ is the smallest.
Since $\tr(\Pi_L \Phi(\Pi_R)) = \tr(\Phi^*(\Pi_L)\Pi_R)$, $L$ is the minimum shrunk subspace of $\Phi^*$.

Therefore, we obtain the following theorem.

\begin{theorem}\label{thm:King-maximizer}
    For a quiver representation $V$ of an acyclic quiver $Q$ with the dimension vector $\alpha$ with Gaussian integer entries and a weight $\sigma$, the minimum and maximum maximizers of King's criterion can be found in $\poly(|Q|, |\alpha|, |\sigma|, b)$ time, where $b$ denotes the bit complexity of $V$.
\end{theorem}

\section{Harder-Narasimhan filtration and principal partition of quiver representation}\label{sec:HN}
In this section, we introduce the principal partition of a quiver representation based on parametric submodular function minimization and show that it coincides with the HN-filtration.

Let us recall the definition.
Let $\sigma \in \Z^{Q_0}$ be a weight and $\tau \in \Z_+^{Q_0}$ a strictly monotone weight, i.e., a nonnegative weight such that $\tau(\dimv W) > 0$ if $W \neq \{0\}$.
We define the \emph{slope} of a quiver nonzero representation $V$ as $\mu(V) = \sigma(\dimv V)/\tau(\dimv V)$.
We say that $V$ is \emph{$\mu$-semistable} if $\mu(W) \leq \mu(V)$ for any nonzero subrepresentation $W$ of $V$.
The HN-filtration of $V$ is a unique filtration $\{0\} = W_0 < W_1 < \cdots < W_k = V$ such that (i) $\mu(W_i/W_{i-1}) > \mu(W_{i+1}/ W_i)$ for $i \in [k-1]$ and (ii) $W_i/W_{i-1}$ is $\mu$-semistable.

\subsection{Equivalence to principal partition}
Let $V$ be a quiver representation and $\lambda \in \R$ a parameter.
Define a parametric modular function $f_\lambda$ as
\begin{align}
    f_\lambda(W) \coloneqq \lambda \tau(\dimv W) - \sigma(\dimv W).
\end{align}
Let $\caL(\lambda)$ denote the modular lattice of the minimizers of $f_\lambda$.

\begin{lemma}
    Let $\lambda > \lambda'$. If $W \in \caL(\lambda)$ and $W' \in \caL(\lambda')$, then $W \leq W'$.
\end{lemma}
\begin{proof}
    For notational simplicity, we abbreviate $\sigma(\dimv W)$ and $\tau(\dimv W)$ as $\sigma(W)$ and $\tau(W)$, respectively.
    Using the modularity, we have
    \begin{align}
        f_\lambda(W) + f_{\lambda'}(W') 
        &= \lambda \tau(W) + \lambda' \tau(W') - \sigma(W) - \sigma(W') \\
        &= \lambda' (\tau(W) + \tau(W')) - \sigma(W) - \sigma(W') + (\lambda - \lambda') \tau(W) \\
        &= \lambda' (\tau(W + W') + \tau(W \cap W')) - \sigma(W + W') - \sigma(W \cap W') + (\lambda - \lambda') \tau(W) \\
        &= f_{\lambda'}(W + W')  + f_{\lambda}(W \cap W') + (\lambda - \lambda') (\tau(W) - \tau(W \cap W')) \\
        &= f_{\lambda'}(W + W')  + f_{\lambda}(W \cap W') + (\lambda - \lambda') \tau(W/(W \cap W')) \\
        &\geq f_{\lambda'}(W')  + f_{\lambda}(W), 
    \end{align}
    where the last inequality follows since $\tau(W / (W \cap W')) \geq 0$ by the assumptions on $\tau$ and $W, W'$ are minimizers of $f_\lambda$ and $f_{\lambda'}$, respectively.
    Therefore, the inequality is tight.
    Since $\lambda > \lambda'$ and $\tau$ is strictly monotone, $W / (W \cap W')$ must be the zero representation, which implies $W \leq W'$.
\end{proof}

\begin{remark}\label{rem:strict-monotone}
    As shown in the above argument, the positivity of $\tau$ is required only for $W / (W \cap W')$ for the \emph{minimizers} $W$ and $W'$.
    This slightly relaxed condition is necessary for the coarse DM-decomposition in \Cref{subsec:coarse-DM}.
\end{remark}

Let $W^-(\lambda)$ and $W^+(\lambda)$ be the minimum and maximum minimizers of $f_\lambda$, respectively.
By the lemma, $W^+(\lambda) \leq W^-(\lambda')$ for $\lambda > \lambda'$.
There must be a finite set of $\lambda$ such that $\caL(\lambda)$ consists of more than one element.
We call such a value of $\lambda$ a \emph{critical value}.
Let $\lambda_1 > \cdots > \lambda_k$ be the critical values.
Then, they induce a filtration
\begin{align}
    \{0\} = W^-(\lambda_1) < W^+(\lambda_1) = W^-(\lambda_2) < W^+(\lambda_2) =  \cdots = W^-(\lambda_k) < W^+(\lambda_k) = V.
\end{align}
We show that this filtration satisfies the definition of the HN-filtration.

\begin{theorem}
    For each $i \in [k]$,
    \begin{itemize}
        \item $\mu(W^+(\lambda_i) / W^-(\lambda_i)) = \lambda_i$, and
        \item $W^+(\lambda_i) / W^-(\lambda_i)$ is $\mu$-semistable.
    \end{itemize}
    Thus, the above filtration coincides with the HN-filtration.
\end{theorem}
\begin{proof}
    Since $W^+(\lambda_i)$ and $W^-(\lambda_i)$ are both minimizers of $f_{\lambda_i}$, we have $f_{\lambda_i}(W^+(\lambda_i)) = f_{\lambda_i}(W^-(\lambda_i))$.
    Hence
    \begin{align}
        0 &= \lambda_i (\tau(W^+(\lambda_i)) - \tau(W^-(\lambda_i))) - \sigma(W^+(\lambda_i)) - \sigma(W^-(\lambda_i)) \\
        &= \lambda_i \tau(W^+(\lambda_i) / W^-(\lambda_i)) - \sigma(W^+(\lambda_i) / W^-(\lambda_i)).
    \end{align}
    Thus, $\mu(W^+(\lambda_i) / W^-(\lambda_i)) = \lambda_i$.

    For the second item, let $W' \leq W^+(\lambda_i) / W^-(\lambda_i)$ be a nonzero subrepresentation of $W^+(\lambda_i) / W^-(\lambda_i)$.
    Then, $W' + W^-(\lambda_i)$ is a subrepresentation of $V$.
    Since $W^- (\lambda_i)$ is a minimizer of $f_{\lambda_i}$, we have
    \begin{align}
        \lambda_i \tau(W' + W^-(\lambda_i)) - \sigma(W' + W^-(\lambda_i)) \geq \lambda_i \tau(W^- (\lambda_i)) - \sigma(W^- (\lambda_i)).
    \end{align}
    By the modularity of $\sigma$ and $\tau$, we have
    $
        \lambda_i \tau(W') - \sigma(W') \geq 0,
    $
    which implies $\mu(W') \leq \lambda_i = \mu(W^+(\lambda_i) / W^-(\lambda_i))$.
    Thus, $W^+(\lambda_i) / W^-(\lambda_i)$ is $\mu$-semistable.
\end{proof}

\subsection{Algorithm for Harder-Narasimhan filtration}
The above observation yields an efficient algorithm for finding the HN-filtration.
First, we reduce the $\mu$-semistablity checking to $\sigma$-semistability checking.

\begin{lemma}\label{lem:slope-to-weight}
    Let $\mu = \sigma/\tau$ be a slope and $V$ a quiver representation with $\mu(V) = p/q$, where $p$ is an integer and $q$ is a positive integer.
    Then, $V$ is $\mu$-semistable if and only if $V$ is $\sigma'$-semistable for $\sigma' = q\sigma - p \tau$.
\end{lemma}
\begin{proof}
    By definition, $V$ is $\mu$-semistable if and only if
    \begin{align}
        \mu(V)\tau(W) - \sigma(W) \geq \mu(V)\tau(V) - \sigma(V) = 0
    \end{align}
    for any subrepresentation $W \leq V$.
    This is equivalent to 
    \begin{align}
        (q\sigma - p\tau)(W)  \leq (q\sigma - p\tau)(V) = 0,
    \end{align}
    which is King's criterion for the $\sigma'$-semistability of $V$.
\end{proof}

Let 
\begin{align}
    S \coloneqq \{ p / q : p \in [-\sigma^-(\alpha), \sigma^+(\alpha)],\ q \in [\tau(\alpha)] \}
\end{align}
be the set of possible slope values.
Note that $\abs{S} = O(|\sigma||\tau|)$.
For each slope $\lambda = p/q$, one can compute the maximum minimizer of $f_\lambda$ by \Cref{thm:King-maximizer} for $\sigma'$-semistability.
Thus, we obtain the HN-filtration by $O(|\sigma||\tau|)$ computations of the maximizers of King's criterion.
A pseudocode is given in \Cref{alg:HN}.

\begin{algorithm}
    \caption{Algorithm for finding the HN-filtration.\label{alg:HN}}
\begin{algorithmic}[1]
    \State Set $i \gets 0$, $W_0 \gets \{0\}$.
    \For{each possible slope $\lambda = p/q \in S$ in the decreasing order}
    \State Invoke the algorithm of \Cref{thm:King-maximizer} with weight $\sigma' = q\sigma - p \tau$ to find the maximum minimizer $W$ of $f_\lambda$.
    \If{$W_i < W$} 
    \State $W_{i+1} \gets W$ and $i \gets i + 1$.
    \EndIf
    \EndFor
    \State \Return $(W_0, \dots, W_i)$.
\end{algorithmic}
\end{algorithm}

\begin{theorem}\label{thm:HN}
    Let $V$ be a quiver representation of an acyclic quiver $Q$ with the dimension vector $\alpha$ with Gaussian integer entries.
    Let $\mu = \sigma/\tau$ be a slope.
    Then, \Cref{alg:HN} finds the HN-filtration in $\poly(|Q|, |\alpha|, |\sigma|, |\tau|, b)$ time, where $b$ is the bit complexity of $V$.
\end{theorem}

\subsection{Relation to coarse DM-decomposition}\label{subsec:coarse-DM}
The DM-decomposition is the decomposition of a bipartite graph into smaller graphs that represent the structure of all maximum matchings and minimum vertex covers.
It is well known that the DM-decomposition corresponds to a maximal chain of minimizers of the surplus function.
Also, this is the finest decomposition of generic matrices under permutation of rows and columns.
See \cite{Murota2009} for the details.

Generalizing the DM-decomposition, \citet{Hirai2024} introduced the \emph{coarse DM-decomposition} for linear matrices.
Let $A = \sum_k x_k A_k$ be a linear matrix and $\caA$ the corresponding matrix space.
Without loss of generality, we can assume that $A_k$ have no common element in their kernels;
otherwise, we can simply delete the zero columns of $A$ corresponding to the common kernel space.

Consider a polytope in $\R^2$ defined by 
\begin{align}
    \conv\{(\dim X, \dim Y) : \text{$(X, Y)$ independent subspace of $A$} \}.
\end{align}
Then, the extreme points of the polytope other than the origin correspond to a flag in the lattice of the independent subspace:
\begin{align}
    \C^n  &= X_0 > X_1 > \dots > X_\ell = \{0\}, \\
    \{0\} &= Y_0 < Y_1 < \dots < Y_\ell = \C^n.
\end{align}
By elementary transformation, one can assume that both $X_i$ and $Y_i$ are coordinate subspaces of the columns and rows, respectively.
The coarse DM-decomposition is given by the decomposition of the rows and columns of $A$ into 
$
    (X_{i-1} \setminus X_i, Y_i \setminus Y_{i-1})
$
for $i = 1, \dots, \ell$.
 
Here, we show that the coarse DM-decomposition is a HN-filtration for the generalized Kronecker quiver.
Let us regard $A$ as a representation $V$ of the generalized Kronecker quiver.
Consider weights $\sigma = (1,0)$ and $\tau = (0,1)$. 
The HN-filtration is given by the minimizers of a parameterized modular function
\begin{align}
    f_\lambda(W) = \lambda \dim W(2) - \dim W(1),
\end{align}
where $\lambda \in \R$ and $W \leq V$.
If $\lambda < 0$,
then the unique minimizer of $f_\lambda$ is $(\C^n, \C^n)$.
Suppose that $\lambda \geq 0$.
Since $\caA W(1) \leq W(2)$ for any subrepresentation $W$, we can assume that $W(2) = \caA W(1)$ for any minimizer of $f_\lambda$.
By the assumption that $A$ has no common kernel element, $W(2) \neq \{0\}$ if $W(1) \neq \{0\}$.
So $\tau$ is strictly monotone for minimizers of $f_\lambda$ (see \Cref{rem:strict-monotone}).
Therefore, it suffices to consider a parameterized submodular function
\begin{align}
    f_\lambda(U) = \lambda \dim \caA U - \dim U,
\end{align}
where $\lambda \geq 0$ and $U \leq \C^n$ is a subspace.
Let $\lambda_1 > \cdots > \lambda_k \geq 0$ be the critical values.
As before, denote by $U_i^-$ and $U_i^+$ the minimum and maximum minimizers of $f_{\lambda_i}$, respectively.
They form a flag in the column space, as follows:
\begin{align}
    \{0\} = U_1^- < U_1^+ = U_2^- < U_2^+ = \cdots = U_k^- < U_k^+ = \C^n.
\end{align}
Rename the flag as $\C^n = Y_0 > Y_1 > \dots > Y_k = \{0\}$ to match the notation of the coarse DM-decomposition.
Let $X_i = \caA(Y_i)^\perp$ ($i = 0, 1, \dots, k$) be a flag in the row space,
where $Z^\perp$ denotes the orthogonal complement of a vector space $Z$.
Then, $(X_i, Y_i)$ is an independent subspace of $A$ by construction.
In fact, the above flags are exactly the flags used in the coarse DM-decomposition.
To see this, let us consider an independent subspace $(X, Y)$ that corresponds to an extreme point of the polytope.
As $(\dim X, \dim Y)$ is extreme, there exists a slope $(\lambda, 1)$ that exposes it.
That is, 
\begin{align}
    \lambda \dim X + \dim Y \geq \lambda \dim X' + \dim Y'
\end{align}
for any independent $(X', Y')$.
Substituting $X = \caA(Y)^\perp$, 
\begin{align}
    \lambda (n - \dim \caA(Y)) + \dim Y \geq \lambda (n - \dim \caA(Y')) + \dim Y',
\end{align}
which shows that $Y$ minimizes $f_\lambda$ and vice versa.

Therefore, \Cref{alg:HN} finds the coarse DM-decomposition of $A$.
Since $\abs{\tau} = \abs{\sigma} = 1$, \Cref{alg:HN} runs in polynomial time.

\begin{remark}
    \cite{Hirai2024} defined the coarse DM-decomposition for any field.
    Since the underlying quiver is the generalized Kronecker quiver, finding the maximizers of King's criterion is exactly finding the minimum shrunk subspaces of an approximate matrix space.
    Therefore, instead of modified operator Sinkhorn iteration, one can use an algebraic algorithm of \cite{Ivanyos2018} which can work in any field.
\end{remark}

\section{Rank-one representation}\label{sec:rank-one}
In this section, we particularly focus on a rank-one representation $V$ of an acyclic quiver $Q = (Q_0, Q_1)$, i.e.,
for each arc $a \in Q_1$, the corresponding linear map $V(a)$ is rank-one.
For a weight $\sigma \in \Z^{Q_0}$,
we present a simple combinatorial characterization of the $\sigma$-semistability of $V$ in terms of linear matroids
and show that
the $\sigma$-semistability of $V$ is equivalent to
the feasibility of a certain instance of the \emph{submodular flow problem} arising from $(V, \sigma)$
(with easily verifiable additional conditions).
In particular, these imply that one can check the $\sigma$-semistability of a rank-one representation in strongly polynomial time.

\subsection{Preliminarlies on matroid and submodular flow}
We first recall the basics of matroids.
A \emph{matroid} (e.g.,~\cite{Oxley2011}) is a pair $\M = (S, \mathcal{B})$ of a finite set $S$ and a nonempty family $\mathcal{B} \subseteq 2^S$ of subsets satisfying the following exchange axiom: for any $B, B' \in \mathcal{B}$ and $x \in B \setminus B'$,
there exists $x' \in B' \setminus B$ such that $B \setminus \{x\} \cup \{x'\} \in \mathcal{B}$.
The family $\mathcal{B}$ is referred to as the \emph{base family} of $\M$
and each member in $\mathcal{B}$ a \emph{base}.
For a matroid $\M = (S, \mathcal{B})$,
its \emph{rank function} $r : 2^S \to \Z$ is defined by $r(X) \coloneqq \max \{ B \cap X : B \in \mathcal{B} \}$.
It is well-known that the rank function is submodular, i.e.,
$r(X) + r(Y) \geq r(X \cap Y) + r(X \cup Y)$ for $X, Y \subseteq S$.

\begin{example}[linear matroid]\label{ex:linear-matroid}
    One of the most fundamental examples of matroids is a \emph{linear matroid}, which arises from a matrix, or equivalently, a finite multiset of vectors, as follows.
Let $U$ be a vector space and $S \subseteq U$ a finite multiset of vectors.
Then, the pair $(S, \mathcal{B})$ in which
\begin{align}
    \mathcal{B} \coloneqq \{ X \subseteq S : |X| = \dim \spnn{X} = \dim \spnn{S} \}
\end{align}
forms a matroid; we refer to this as a \emph{linear matroid generated by $S$}.
The rank function $r$ of a linear matroid is simply the map $X \mapsto \dim \spnn{X}$.
\end{example}

We then prepare notation
before introducing the submodular flow problem.
For a submodular function $f : 2^S \to \R \cup \{+\infty\}$ with $f(\emptyset) = 0$ and $f(S) < +\infty$,
its \emph{base polyhedron} $\BB(f)$ is a polyhedron defined by
\begin{align}
    \BB(f) \coloneqq \{ x \in \R^S : x(S) = f(S),\  x(X) \leq f(X) \ (\forall X \subseteq S) \}.
\end{align}
The following are fundamental to the theory of submodular functions:
\begin{lemma}[{\cite[Chapter~4]{Murota2003}}]\label{lem:submodular-polyhedron}
    Let $f : 2^{S_1} \to \R \cup \{+\infty\}$ and $g : 2^{S_2} \to \R \cup \{+\infty\}$ be submodular functions satisfying that $f(\emptyset) = g(\emptyset) = 0$ and $f(S_1)$ and $g(S_2)$ are finite.
    \begin{enumerate}[{label={\textup{(\arabic*)}}}]
        \item If $f$ is integer-valued,
        then $\BB(f)$ is an integral polyhedron.
        \item If $S_1 = S_2$, then $\BB(f + g) = \BB(f) + \BB(g)$, where ``$+$'' in the right-hand side denotes the Minkowski sum.
        If $S_1$ and $S_2$ are disjoint, then $\BB(f + g) = \BB(f) \times \BB(g)$.
    \end{enumerate}
\end{lemma}

\begin{example}[base polytope of a matroid rank function]\label{ex:base-polyhedron}
Let $r$ be the rank function of a matroid $\M = (S, \mathcal{B})$
and $(-r)^\#$ the \emph{dual}~\cite[Section~2.3]{Fujishige2005} of $-r$, i.e.,
\begin{align}
    (-r)^\#(X) \coloneqq r(S \setminus X) - r(S);
\end{align}
both $r$ and $(-r)^\#$ are submodular.
Then, the base polytopes $\BB(r)$ and $\BB((-r)^\#)$ are represented as
\begin{align}
    \BB(r) = \conv \left\{ \chi_{B} : B \in \mathcal{B} \right\}, \qquad \BB((-r)^\#) = \conv \left\{ -\chi_{B} : B \in \mathcal{B} \right\},
\end{align}
where $\chi_B \in \{0,1\}^S$ denotes the characteristic vector of $B \subseteq S$,
i.e.,
$\chi_B(s) = 1$ if $s \in S$ and $\chi_B(s) = 0$ if $s \in S \setminus B$.

In addition, let $k \in \Z_+$ be a nonnegative integer.
Then, by \Cref{lem:submodular-polyhedron}~(2), we have
\begin{align}
    \BB(kr) &= \conv \left\{ \sum_{\ell = 1}^k \chi_{B_\ell} : B_1, \dots, B_k \in \mathcal{B} \right\},\\
    \BB((-kr)^\#) &= \conv \left\{ - \sum_{\ell = 1}^k \chi_{B_\ell} : B_1, \dots, B_k \in \mathcal{B} \right\}.
\end{align}
Moreover, by \Cref{lem:submodular-polyhedron}~(1), the set of integral points in $\BB(kr)$ (resp.\ $\BB((-kr)^\#)$) consists of $\sum_{i = 1}^k \chi_{B_i}$ (resp.\ $- \sum_{i = 1}^k \chi_{B_i}$) for $B_1, \dots, B_k \in \mathcal{B}$, that is,
\begin{align}
    \BB(kr) \cap \Z^S &= \left\{ \sum_{\ell = 1}^k \chi_{B_\ell} : B_1, \dots, B_k \in \mathcal{B} \right\},\\
    \BB((-kr)^\#) \cap \Z^S &= \left\{ - \sum_{\ell = 1}^k \chi_{B_\ell} : B_1, \dots, B_k \in \mathcal{B}\right\}. \qedhere
\end{align}
\end{example}

In (the feasibility version of) the \emph{submodular flow problem}~\cite{Edmonds1977} (see also~\cite[Section~5.1]{Fujishige2005}),
given a directed graph $D = (S, A)$, an upper capacity function $\overline{c} : A \to \R \cup \{+\infty\}$, a lower capacity function $\underline{c} : A \to \R \cup \{-\infty\}$, and a submodular function $f : 2^S \to \R \cup \{+\infty\}$ with $f(\emptyset) = f(S) = 0$,
we are asked to find a \emph{feasible submodular flow}, namely,
a flow (vector) $\varphi \in \R^A$ such that
$\underline{c}(a) \leq \varphi(a) \leq \overline{c}(a)$ for each $a \in A$
and its boundary $\partial \varphi$ belongs to the base polyhedron $\BB(f)$,
if it exists.
We say that an instance $(D, \overline{c}, \underline{c}, f)$ is \emph{feasible} if there is a feasible submodular flow for $(D, \overline{c}, \underline{c}, f)$.

The following is well-known in combinatorial optimization.
\begin{theorem}[{\cite{Frank1984}; see also~\cite[Section~5.3]{Fujishige2005}}]\label{thm:chara:submodular-flow}
Let $D = (S, A)$ be a directed graph, $\overline{c} : A \to \R \cup \{+\infty\}$ an upper capacity, $\underline{c} : A \to \R \cup \{-\infty\}$ a lower capacity, and $f : 2^S \to \R \cup \{+\infty\}$ a submodular function with $f(\emptyset) = f(S) = 0$.
Then,
$(D, \overline{c}, \underline{c}, f)$ is feasible if and only if
\begin{align}
\overline{c}(\Out(X)) - \underline{c}(\In(X)) + f(V \setminus X) \geq 0
\end{align}
holds for all $X \subseteq S$.
In particular, if $f$ is integer-valued,
then there exists an integral feasible submodular flow $\varphi \in \Z^A$ whenever it is feasible.
\end{theorem}

\subsection{Reduction from semistability to submodular flow}
Let $Q = (Q_0, Q_1)$ be an acyclic quiver
and $V$ a rank-one representation of $Q$.
Then for each arc $a \in Q_1$,
the corresponding linear map $V(a)$ is representable as $v_a f_a$ for some nonzero vector $v_a \in V(ha)$ and nonzero dual vector $f_a \in V(ta)^*$,
i.e.,
$V(a)$ is the map $V(ta) \ni u \mapsto v_af_a(u) \in V(ha)$.
Building on this representation,
we consider that
the data of $V$ consist of
\begin{itemize}
    \item a vector space $V(i)$ for each vertex $i \in Q_0$,
    \item a nonzero vector $v_a \in V(i)$ for each vertex $i \in Q_0$ and incoming arc $a \in \In(i)$, and
    \item a nonzero dual vector $f_a \in V(i)^*$ for each vertex $i \in Q_0$ and outgoing arc $a \in \Out(i)$.
\end{itemize}
Let $S_i^+ \coloneqq \{ f_a : a \in \Out(i) \}$ and $S_i^- \coloneqq \{ v_a : a \in \In(i) \}$ for $i \in Q_0$.
Note that $S_i^+$ and $S_i^-$ may be multisets;
even if $f_a$ and $f_b$ are the same vectors for distinct $a,b \in \Out(i)$,
we distinguish $f_a$ from $f_b$ in $S_i^+$.

We can easily observe the following from the definition of a subrepresentation.
\begin{lemma}\label{lem:subrep}
    A tuple $(W(i))_{i \in Q_0}$ of vector subspaces $W(i) \leq V(i)$ induces a subrepresentation $W$ of $V$
    if and only if for any arc $a \in Q_1$ with
    $W(ta) \not\leq \ker f_a$, we have $\spnn{v_a} \leq W(ha)$.
\end{lemma}

The above observation (\Cref{lem:subrep}) gives rise to the directed graph $D[V] = (S, A)$ defined by
\begin{align}
    &S \coloneqq \bigcup_{i \in Q_0} \left(S_i^+ \cup S_i^-\right),\\
    &A \coloneqq \left\{ (f_a, v_a) : a \in Q_1 \right\} \cup \left\{ (v_a, f_b) : i \in Q_0,\ v_a \in S_i^-, \ f_b \in S_i^+,\ v_a \notin \ker f_b \right\}.
\end{align}
Intuitively, an arc in $A$ of the form $(f_a, v_a)$ represents the original arc $a \in Q_1$ of the quiver, whereas that of the form $(v_a, f_b)$ represents ``if $W(ha) = W(tb)$ contains $v_a$ then $W(tb) \not\leq \ker f_b$; hence, $W(hb)$ must contain $v_b$ for $W$ to be a subrepresentation of $V$.''
See \Cref{fig:D[V]}.
\begin{figure}
    \centering
    \includegraphics[width=0.8\textwidth]{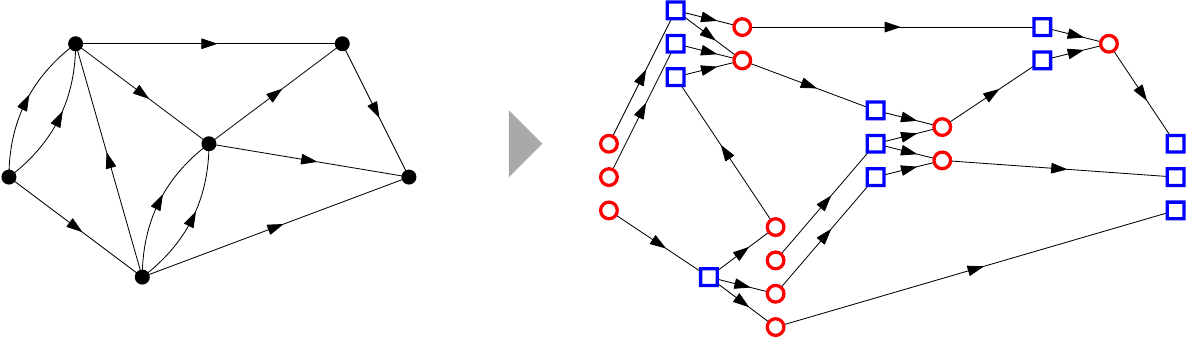}
    \caption{The left is an original quiver $Q$ and the right is the directed graph $D[V]$ constructed from $Q$ and a rank-one representation $V$ of $Q$. The red circles and blue squares in the right graph represent the vertices of $D[V]$ corresponding to the outgoing arcs (endowed with nonzero dual vector) and incoming arcs (endowed with nonzero vector) in $Q$, respectively.\label{fig:D[V]}}
\end{figure}

This digraph $D[V]$ enables us to characterize the $\sigma$-semistability of $V$ in a combinatorial manner.
For each $i \in Q_0$,
let $\M_i^+ = (S_i^+, \mathcal{B}_i^+)$ (resp.\ $\M_i^- = (S_i^-, \mathcal{B}_i^-)$) denote the linear matroid generated by $S_i^+$ (resp.\ $S_i^-$).
Moreover, let $r_i^+$ (resp.\ $r_i^-$) denote the rank function of $\M_i^+$ (resp.\ $\M_i^-$).
Then, we obtain the following.
\begin{theorem}\label{thm:King:rank-one}
Let $V$ be a rank-one representation of $Q$ and $\sigma \in \Z^{Q_0}$ a weight.
Then $V$ is $\sigma$-semistable if and only if
\begin{enumerate}[{label={\upshape{(K\arabic*)}}}]
    \item $\sum_{i \in Q_0} \sigma^+(i) \dim V(i) = \sum_{i \in Q_0} \sigma^-(i) \dim V(i) \eqqcolon \Sigma$ and
    \item for any lower set $X \subseteq S$ of $D[V]$, we have
    \begin{align}
    \sum_{i \in Q_0}\left( \sigma^+(i) r_i^+(S_i^+ \setminus X) + \sigma^-(i) r_i^-(S_i^- \cap X) \right) \geq \Sigma.
\end{align}
\end{enumerate}
\end{theorem}
\begin{proof}
By King's criterion,
$V$ is $\sigma$-semistable if and only if
$\sigma (\dimv V) = 0$ and $\sigma (\dimv W) \leq 0$ for any subrepresentation $W$ of $V$.
The former condition $\sigma (\dimv V) = 0$ of King's criterion is equivalent to (K1).
Hence, it suffices to see the equivalence between
the latter condition of King's criterion and (K2)
under the condition (K1).

Suppose that the latter condition of King's criterion holds, i.e., $\sigma (\dimv W) \leq 0$ for any subrepresentation $W$ of $V$.
Take any lower set $X \subseteq S$ of $D[V]$.
We define from $X$, the tuple $(W(i))_{i \in Q_0}$ of vector subspaces $W(i) \leq V(i)$ by
\begin{align}
    W(i) \coloneqq
    \begin{cases}
        \bigcap \{ \ker f_b : f_b \in S_i^+ \setminus X \} & \text{if $\sigma(i) \geq 0$},\\
        \spnn{v_a : v_a \in S_i^- \cap X} & \text{if $\sigma(i) < 0$}.
    \end{cases}
\end{align}

We prove that $(W(i))_{i \in Q_0}$ induces a subrepresentation of $V$.
By \Cref{lem:subrep},
it suffices to see that for any arc $a \in Q_1$ with $\spnn{v_a} \not\leq W(ha)$, or equivalently, $v_a \notin W(ha)$,
we have $W(ta) \leq \ker f_a$.
Since $X$ is a lower set,
there is no arc between $v_a \in S_i^- \cap X$ and $f_b \in S_i^+ \setminus X$,
i.e., every $v_a \in S_i^- \cap X$ is in $\ker f_b$.
Hence, for any $i \in Q_0$, we have $\spnn{v_a : v_a \in S_i^- \cap X} \leq \bigcap \{ \ker f_b : f_b \in S_i^+ \setminus X \}$,
which implies $\spnn{v_a : v_a \in S_i^- \cap X} \leq W(i)$.
Take any arc $a \in Q_1$ with $v_a \notin W(ha)$.
Then, $v_a \in S_{ha}^- \setminus X$ by the above argument,
and hence $f_a \in S_{ta}^+ \setminus X$ as $X$ is a lower set.
Therefore, $W(ta) \leq \ker f_a$ holds.

Since $(W(i))_{i \in Q_0}$ induces a subrepresentation of $V$,
we have $\sigma(\dimv W) \leq 0$,
i.e.,
\begin{align}\label{eq:King}
    \sum_{i \in Q_0} \left( \sigma^+(i) \dim \bigcap \{ \ker f_b : f_b \in S_i^+ \setminus X \} - \sigma^-(i) \dim \spnn{v_a : v_a \in S_i^- \cap X} \right) \leq 0.
\end{align}
Since
$\dim \bigcap \{ \ker f_b : f_b \in S_i^+ \setminus X \} = \dim V(i)^* - \dim \spnn{f_b : f_b \in S_i^+ \setminus X} = \dim V(i) - r_i^+(S_i^+ \setminus X)$
and $\dim \spnn{v_a : v_a \in S_i^- \cap X} = r_i^-(S_i^- \cap X)$ (see \Cref{ex:linear-matroid} for the rank function of a linear matroid),
the inequality~\eqref{eq:King} is equivalent to
\begin{align}\label{eq:K2}
    \sum_{i \in Q_0}\left( \sigma^+(i) r_i^+(S_i^+ \setminus X) + \sigma^-(i) r_i^-(S_i^- \cap X) \right) \geq \Sigma,
\end{align}
where we recall that $\Sigma = \sum_{i \in Q_0} \sigma^+(i) \dim V(i)$ by (K1).

Conversely, suppose that (K2) holds.
Take any subrepresentation $W$ of $V$.
Our aim is to prove $\sigma(\dimv W) \leq 0$.

We define a vertex subset $X \subseteq S$ of $D[V]$ by
\begin{align}
    X \coloneqq \bigcup_{i \in Q_0} \left(\{ v_a \in S_i^- : v_a \in W(i) \} \cup \{ f_b \in S_i^+ : W(i) \not\leq \ker f_b \} \right).
\end{align}
Then, $X$ forms a lower set of $D[V]$.
Indeed, if $f_b \in S_{tb}^+ \cap X$, or equivalently, $W(tb) \not\leq \ker f_b$,
then $W(hb)$ contains $v_b$ since $W$ is a subrepresentation of $V$, implying $v_b \in S_{hb}^- \cap X$.
If $v_a \in S_{ha}^- \cap X$, that is, $v_a \in W(ha)$, then $W(ha) \not\leq \ker f_b$ for $(v_a, f_b) \in A$ by the definition of arc $(v_a, f_b)$; thus we obtain $f_b \in S_{tb}^+ \cap X (= S_{ha}^+ \cap X)$ for any $(v_a, f_b) \in A$.

By reversing the argument from~\eqref{eq:King} to~\eqref{eq:K2} above,
we obtain the inequality~\eqref{eq:King} for this $X$.
Since $\spnn{v_a : v_a \in S_i^- \cap X} \leq W(i) \leq \bigcap \{ \ker f_b : f_b \in S_i^+ \setminus X \}$ and $\sigma^+(i), \sigma^-(i) \geq 0$ for $i \in Q_0$,
the LHS of~\eqref{eq:King} is at least $\sigma(\dimv W)$.
Therefore, we conclude that $\sigma(\dimv W) \leq 0$.
\end{proof}

\Cref{thm:King:rank-one} implies the following necessary conditions for $\sigma$-semistability.
\begin{corollary}\label{cor:necessary}
    Let $V$ be a rank-one representation of $Q$ and $\sigma \in \Z^{Q_0}$ a weight.
    If $V$ is $\sigma$-semistable,
    then {\rm (K1)} and the following full-dimensional condition {\rm (F)} hold:
    \begin{enumerate}[{label={\upshape{(F)}}}]
        \item $r_i^+(S_i^+) = \dim V(i)$ for $i \in Q_0^+$ and $r_i^-(S_i^-) = \dim V(i)$ for $i \in Q_0^-$.
    \end{enumerate}
\end{corollary}
\begin{proof}
    We clearly have $r_i^+(S_i^+) \leq \dim V(i)^* = \dim V(i)$ and $r_i^-(S_i^-) \leq \dim V(i)$ for each $i \in Q_0$.
    Hence we obtain $\sum_{i \in Q_0} \sigma^+(i) r_i^+(S_i^+) \leq \Sigma$ and $\sum_{i \in Q_0} \sigma^-(i) r_i^-(S_i^-) \leq \Sigma$.

    Suppose that $V$ is $\sigma$-semistable.
    Then, by \Cref{thm:King:rank-one},
    we have (K1) and $\sum_{i \in Q_0} \sigma^+(i) r_i^+(S_i^+) \geq \Sigma$ (corresponding to the lower set $X = \emptyset$ in (K2))
    and $\sum_{i \in Q_0} \sigma^-(i) r_i^-(S_i^-) \geq \Sigma$ (corresponding to the lower set $X = S$ in (K2)).
    Thus we obtain $\sum_{i \in Q_0} \sigma^+(i) r_i^+(S_i^+) = \sum_{i \in Q_0} \sigma^-(i) r_i^-(S_i^-) = \Sigma$,
    which implies $r_i^+(S_i^+) = \dim V(i)$ for $i \in Q_0^+$ and $r_i^-(S_i^-) = \dim V(i)$ for $i \in Q_0^-$.
\end{proof}

In the following, we see that the combinatorial characterization of the $\sigma$-semistability given in \Cref{thm:King:rank-one} can be rephrased as the feasibility of a certain instance of the submodular flow problem on $D[V]$.
Here, we assume that the conditions (K1)
and (F) hold (recall \Cref{cor:necessary}).
Set $\overline{c}$ and $\underline{c}$ as $\overline{c}(a) \coloneqq +\infty$ and $\underline{c}(a) \coloneqq 0$ for $a \in A$,
denoted as $\overline{\infty}$ and $\underline{0}$, respectively.
We define a function $f_V : 2^S \to \R$ by
\begin{align}
    f_V(X) \coloneqq& \sum_{i \in Q_0} \left(\sigma^+(i) r_i^+(S_i^+ \cap X) +  (-\sigma^-(i) r_i^-)^\# (S_i^- \cap X) \right)\\
    {}=& \sum_{i \in Q_0} \left(\sigma^+(i) r_i^+(S_i^+ \cap X) +  \left(\sigma^-(i) r_i^-(S_i^- \setminus X) - \sigma^-(i)r_i^-(S_i^-)\right) \right)\\
    {}=& \sum_{i \in Q_0} \left(\sigma^+(i) r_i^+(S_i^+ \cap X) +  \sigma^-(i) r_i^-(S_i^- \setminus X) \right) - \Sigma \label{eq:f_V}
\end{align}
for $X \subseteq S$,
where the second equality follows from the definition of the dual $(-\sigma^-(i) r_i^-)^\#$ (see \Cref{ex:base-polyhedron})
and the last follows from $\sum_{i \in Q_0} \sigma^-(i) r_i^-(S_i^-) = \Sigma$ by (K1) and (F).
The assumptions (K1) and (F) also imply $f(\emptyset) = f(S) = 0$.
Since the rank functions $r_i^+$ and $r_i^-$ are submodular and $\sigma^+(i)$ and $\sigma^-(i)$ are nonnegative,
the function $f_V$ is also submodular.
Thus $(D[V], \overline{\infty}, \underline{0}, f_V)$ forms an instance of the submodular flow problem.

The following theorem says that the $\sigma$-semistability can be characterized as the feasibility of $(D[V], \overline{\infty}, \underline{0}, f_V)$.
Here, for a finite set $T$, subset $T' \subseteq T$, and $\gamma \in \R^T$,
let $\gamma|_{T'}$ denote the projection of $\gamma$ to $\R^{T'}$.
\begin{theorem}\label{thm:submodular-flow}
Let $V$ be a rank-one representation of $Q$ and $\sigma \in \Z^{Q_0}$ a weight.
Then, $V$ is $\sigma$-semistable if and only if it satisfies {\rm (K1)}, {\rm (F)}, and
\begin{enumerate}[{label={\upshape{(S)}}}]
    \item the instance $(D[V], \overline{\infty}, \underline{0}, f_V)$ of the submodular flow problem is feasible.
\end{enumerate}
In addition, the last condition {\rm (S)} can be replaced by the following more combinatorial condition {\rm (S)'}:
\begin{enumerate}[{label={\upshape{(S)'}}}]
    \item there is a nonnegative integral flow $\varphi \in \Z_+^A$ such that, for each $i \in Q_0$, we have
    $\partial \varphi|_{S_i^+} = \sum_{\ell = 1}^{\sigma^+(i)} \chi_{B_\ell}$ for some $B_1, \dots, B_{\sigma^+(i)} \in \mathcal{B}_i^+$ and $\partial \varphi|_{S_i^-} = - \sum_{\ell = 1}^{\sigma^-(i)} \chi_{B_\ell}$ for some $B_1, \dots, B_{\sigma^-(i)} \in \mathcal{B}_i^-$.
\end{enumerate}
\end{theorem}
\begin{proof}
    We first see that, under (K1) and (F),
    the condition (K2) is equivalent to (S), which implies the former assertion.
    By~\Cref{thm:chara:submodular-flow},
    the rank-one representation $V$ satisfies (S)
    if and only if
    \begin{align}
        \overline{\infty}(\Out(X)) - \underline{0}(\In(X)) + f_V(V \setminus X) \geq 0
    \end{align}
    for any $X \subseteq S$,
    which can be rephrased using~\eqref{eq:f_V} as
    \begin{align}
        \sum_{i \in Q_0} \left( \sigma^+(i) r_i^+(S_i^+(i) \setminus X) + \sigma^-(i) r_i^-(S_i^- \cap X)\right) \geq \Sigma
    \end{align}
    for any lower set $X \subseteq S$ of $D[V]$.
    This is equivalent to the condition (K2).

    It follows from \Cref{lem:submodular-polyhedron}~(2), \Cref{ex:base-polyhedron},
    and the fact that $f_V$ is integer-valued that
    \begin{align}
        \BB(f_V) \cap \Z^S = \prod_{i \in Q_0} \left\{ \sum_{\ell = 1}^{\sigma^+(i)} \chi_{B_\ell} : B_1, \dots, B_{\sigma^+(i)} \in \mathcal{B}_i^+ \right\} \times \left\{ -\sum_{\ell = 1}^{\sigma^-(i)} \chi_{B_\ell} : B_1, \dots, B_{\sigma^-(i)} \in \mathcal{B}_i^- \right\}.
    \end{align}
    From the above and \Cref{thm:chara:submodular-flow},
    we obtain the latter assertion.
\end{proof}

The conditions (K1) and (F) are clearly verifiable in strongly polynomial time.
Since we can minimize the submodular function $X \mapsto \sum_{i \in Q_0}\left( \sigma^+(i) r_i^+(S_i^+ \setminus X) + \sigma^-(i) r_i^-(S_i^- \cap X) \right)$ over the ring family $\{ X : \text{$X$ is a lower set of $D[V]$} \}$ in strongly polynomial time (see~\cite{Jiang2021} and the references therein),
the condition (K2) is also verifiable in strongly polynomial time.
Similarly, one can check (S) in strongly polynomial time,
since the function $X \mapsto \overline{c}(\Out(X)) - \underline{c}(\In(X)) + f(V \setminus X)$ in \Cref{thm:chara:submodular-flow} is submodular~\cite[Section~2.3]{Fujishige2005}; see also~\cite{Frank1984}.
Therefore, we obtain the following:
\begin{theorem}\label{thm:rank-one:poly}
    Let $V$ be a rank-one representation of a quiver $Q$ and $\sigma \in \Z^{Q_0}$ a weight.
    Then, one can check if $V$ is $\sigma$-semistable in strongly polynomial time.
\end{theorem}

\subsection{Implications}

\Cref{thm:submodular-flow} states that King's criterion for a rank-one representation serves as a good characterization for the existence of a feasible submodular flow in a certain instance of the submodular flow problem, or equivalently, a flow such that its boundary can be decomposed as a sum of indicators of matroid bases.
Pursuing this direction, we specializes \cref{thm:submodular-flow} to quivers having specific structures:
generalized Kronecker quivers and star quivers, which arise from linear matroid intersection and rank-one BL polytopes, respectively.
In particular, we observe that \cref{thm:submodular-flow} can recover well-known good characterizations on these problems.

\paragraph{Generalized Kronecker quivers and linear matroid intersction.}

Suppose that $Q = (Q_0, Q_1)$ is a generalized Kronecker quiver, i.e.,
$Q_0$ consists of the two vertices $1, 2$ and $Q_1$ consists of $m$ parallel arcs $a_1, \dots, a_m$ from $1$ to $2$
(recall \Cref{fig:Kr-star-quivers}).
Let $V$ be a rank-one representation of $Q$, in which
$V(a_k) \coloneqq v_k f_k$ for some nonzero $v_k \in V(2)$ and $f_k \in V(1)^*$.
Here, $f_k$ is regarded as a row vector.
That is,
$S^+ \coloneqq S_1^+ = \{ f_1, \dots, f_m \}$ and $S^- \coloneqq S_2^- = \{ v_1, \dotsc, v_m \}$ (these may be multisets).
Let $\M^+ \coloneqq \M_1^+$ (resp.\ $\M^- \coloneqq \M_2^-$) be the linear matroid generated by $S^+$ (resp.\ $S^-$), and $r^+ \coloneqq r_1^+$ (resp.\ $r^- \coloneqq r_2^-$) denotes the rank function of $\M^+$ (resp.\ $\M^-$).
We assume that $\spnn{S^+} = V(1)^*$ and $\spnn{S^-} = V(2)$.
In the following, we naturally identify $S^+$ and $S^-$ with $Q_1$ via the correspondences between $f_i, v_i$ and $a_i$,
that is, we consider that both of the ground sets of $\M^+$ and $\M^-$ are $Q_1$.

Set a weight $\sigma$ as $\sigma = (1,-1)$.
Then, King's criterion is that $\dim V(1) = \dim V(2)$ and $\dim U \le \dim(\sum_{k=1}^m V(a_k) U)$ for any subspace $U \le V(1)$ (see \Cref{ex:nc-rank}).
A feasible integral submodular flow in this setting corresponds to bases of $\M^+$ and $\M^-$ indexed by the same arc sets.
Hence, we obtain the following corollary of \cref{thm:submodular-flow}, which agrees with the well-known characterization of the existence of a common base in linear matroid intersection according to Lov\'{a}sz~\citep{Lovasz1989}.

\begin{corollary}[{\citep{Lovasz1989}}]\label{cor:matroid-intersection}
  Assume $\dim V(1) = \dim V(2)$.
  Then, there is a common base $B \subseteq Q_1$ of $\M^+$ and $\M^-$ if and only if $\dim U \le \dim(\sum_{k=1}^m V(k) U)$ for any subspace $U \le V(1)$.
\end{corollary}

Note that the latter condition in \cref{cor:matroid-intersection} is equivalent to the nc-nonsingularity of a linear symbolic matrix $\sum_{k=1}^m x_k v_k f_k$, as described in \Cref{ex:nc-rank}.

\paragraph{Star quivers and rank-one BL polytopes.}

Suppose that $Q = (Q_0, Q_1)$ is a star quiver, i.e.,
$Q_0 = \{ 0,1, \dots, m \}$ and $Q_1 = \{ (0,1), \dots, (0,m) \}$
(recall \Cref{fig:Kr-star-quivers}).
Let $V$ be a rank-one representation of $Q$, in which
$V((0,i)) \coloneqq v_i f_i$ for some nonzero $v_i \in V(i)$ and $f_i \in V(0)^*$.
That is,
$S_0^+ = \{ f_1, \dots, f_m \}$ (this may be a multiset) and $S_i^- = \{ v_i \}$ for each $i \in [m]$.
Recall that $\M \coloneqq \M_0^+$ is the linear matroid generated by $S_0^+$ and $r \coloneqq r_0^+$ denotes the rank function of $\M$.
We assume $\spnn{S_0^+} = V(0)^*$ and $\dim V(i) = 1$ for $i \in [m]$.

Let $\sigma = (d, -c_1, \dotsc, -c_m) \in \Z^{Q_0}$ be a weight with $d, c_1, \dotsc, c_m > 0$.
Then, as described in \cref{ex:BL polytope}, King's criterion is that $(c_1/d, \dots, c_m/d)$ is in the \emph{Brascamp--Lieb} (BL) \emph{polytope} determined from the rank-one BL-datum $(f_1, \dots, f_m)$, which is the set of points $p \in \R_+^m$ such that
\begin{align}
    \dim U \le \sum_{i=1}^m p_i \dim (f_i U)
\end{align}
for all subspaces $U \le V(0)$.
On the other hand, feasible submodular flows correspond to the base polytope of $\mathbf{M}$.
Therefore, we obtain the following corollary of \cref{thm:submodular-flow}, which recovers the characterization of the rank-one BL polytope by \citet{Barthe1998}.

\begin{corollary}[{\citep{Barthe1998}}]
    The BL polytope associated with the rank-one BL-datum $(f_1, \dots, f_m)$ is the base polytope of $\mathbf{M}$.
\end{corollary}

\section{Polynomial-time semistability testing for general quivers}\label{sec:general-ss}
In this section, we present a polynomial-time algorithm for checking the semistability of a quiver representation of a general quiver, possibly having directed cycles, under the GL-action.\footnote{As mentioned in \cref{subsec:general-quiver}, after submitting the first version of this paper, we were informed by an anonymous reviewer that a similar approach for general quivers has been sketched in \citet[Theorem~10.8 and the last paragraph of Section~10.2]{Mulmuley2017}.}

\subsection{Algebraic condition for semistability}
Let $Q = (Q_0, Q_1)$ be a general quiver and $V$ a representation of $Q$.
The semistability of general quivers can be defined in the same way as that of acyclic quivers, which were given in \cref{sec:semistability}.
Our starting point for checking the semistability is the following Le Bruyn--Procesi theorem.
Note that a path means a directed walk, i.e., a path can visit each vertex many times, as mentioned in \cref{sec:introduction}.

\begin{theorem}[\citet{Bruyn1990}]\label{thm:bruyn-procesi}
    The invariant ring of the action of $\GL(Q, \alpha)$ on the representation space of a quiver $Q$ is generated by 
    \begin{align}
        \tr[V(a_k)\cdots V(a_2) V(a_1)]
    \end{align}
    for a closed path $(a_1, a_2, \dots, a_k)$ of length $k \ge 1$ in $Q$.
    Furthermore, closed paths with length $1 \le k \leq \alpha(Q_0)^2$ generate the invariant ring, where $\alpha = \dimv V$.
\end{theorem}

By \cref{thm:bruyn-procesi}, to check if $V$ is semistable, it suffices to check if $\tr V(C) \neq 0$ for some closed path $C$ with a maximum length of ${\alpha(Q_0)}^2$ in $Q$, where $V(C) \coloneqq V(a_k) \dotsb V(a_2)  V(a_1)$ if $C = (a_1, a_2, \dots, a_k)$.

Recall that a cycle in a digraph can be detected by repeatedly multiplying the adjacency matrix.
Let $A$ be the adjacency matrix of $Q$.
It is well-known that $Q$ has an $i$--$j$ path of length $k$ if and only if $(A^k)_{ij} \neq 0$.
Thus, $Q$ has a cycle of length $k$ if and only if $(A^k)_{ii} \neq 0$ for some vertex $i$.

One can generalize this to quiver representations.
Define the \emph{adjacency matrix} of a quiver representation $V$ as a partitioned matrix
\begin{align}
    A: \bigoplus_{i \in Q_0} \C^{\alpha(i)} \to \bigoplus_{i \in Q_0} \C^{\alpha(i)}
\end{align}
whose $(i,j)$-block is given by 
\begin{align}\label{eq:adj-matrix}
    \sum_{a \in Q_0: ta=j, ha=i} x_a V(a),
\end{align}
where $x_a$ is an indeterminate that is pairwise noncommutative with other indeterminates but commutes with complex numbers.

\begin{lemma}
    Let $V$ be a representation of a quiver $Q$ and $A$ the adjacency matrix of $V$.
    Then, $V$ is semistable if and only if
    \begin{align}\label{def:semistable-polynomial}
        \sum_{k=1}^{{\alpha(Q_0)}^2} \tr A^k
    \end{align}
    is a nonzero polynomial, where $\alpha = \dimv V$.
\end{lemma}
\begin{proof}
    Clearly, for $i \in Q_0$ and $k \in [{\alpha(Q_0)}^2]$, the $(i,i)$-block of $A^k$ equals
    \begin{align}
        \sum_{\text{$C$: closed path of length $k$ starting at $i$}} x^C V(C),
    \end{align}
    where $x^C = x_{a_k} \dotsb x_{a_1}$ if $C = (a_1, \dots, a_k)$.
    Thus, taking the trace within the $(i,i)$-block, we obtain a polynomial
    \begin{align}\label{eq:trace-poly}
        \sum_{\text{$C$: closed path of length $k$ starting at $i$}} x^C \tr V(C),
    \end{align}
    which is nonzero if and only if there is a closed path $C$ of length $k$ starting at $i$ such that $\tr V(C) \neq 0$.
    Thus, by \cref{thm:bruyn-procesi}, $V$ is semistable if and only if the trace of the $(i, i)$-block of $A^k$ is nonzero for some $i \in Q_0$ and $k \in [{\alpha(Q_0)}^2]$.
    It is easily checked that the polynomial~\eqref{def:semistable-polynomial} is the sum of such traces.
    Furthermore, $x^C$ are different monic monomials for different closed paths $C$, and the leftmost factor in $x^C$ corresponds to the first arc that appears in $C$.
    Here, the sum of the distinct traces does not cancel out, proving the claim.
\end{proof}

Therefore, to check the semistability of $V$, it suffices to check if~\eqref{def:semistable-polynomial} is a noncommutative polynomial.
\Cref{sec:deterministic} provides a deterministic algorithm for this by using an identity testing algorithm for noncommutative algebraic branching programs.

\subsection{Deterministic algorithm via white-box polynomial identity testing}\label{sec:deterministic}
A (noncommutative) \emph{algebraic branching program} (ABP) consists of a directed acyclic graph whose vertices are partitioned into $d+1$ parts $L_0, \dotsc, L_d$, each of which is called a \emph{layer}.
The first and last layers $L_0$ and $L_d$ consist of a singleton, and the unique vertices in $L_0$ and $L_d$ are called the \emph{source} and \emph{sink}, respectively.
Arcs may only go from $L_k$ to $L_{k+1}$ for $k = 0, \dotsc, d-1$.
Each arc $a$ is labeled with a homogeneous linear form in noncommutative variables $x_1, \dotsc, x_n$.
The polynomial computed at a vertex $v$ is the sum over all paths, from the source to $v$, of the product of the labeled homogeneous linear forms.
An ABP is said to compute a polynomial $f$ if $f$ is computed at the sink.
The \emph{size} of an ABP means the number of vertices.
Raz and Shpilka~\cite{Raz2005} showed the following result on the identity testing for noncommutative ABPs.

\begin{theorem}[{\cite[Theorem~4]{Raz2005}}]\label{thm:abp-identity-test}
    There is an algorithm that, given a noncommutative ABP of size $s$ in $n$ indeterminates, verifies whether the ABP computes a zero polynomial in time $O(s^5 + sn^4)$.
\end{theorem}

The following lemma constructs a polynomial-sized ABP that computes~\eqref{def:semistable-polynomial}.

\begin{lemma}\label{lem:into-abp}
    There is an ABP of size $O({\alpha(Q_0)}^4)$ that computes the polynomial~\eqref{def:semistable-polynomial}.
\end{lemma}

\begin{proof}
    We construct ${\alpha(Q_0)}^2 + 3$ layers $\{s\} = L_{-1}, L_0, \dotsc, L_{{\alpha(Q_0)}^2}, L_{{\alpha(Q_0)}^2+1} = \{t\}$ as follows.
    Let $s$ and $t$ be the source and sink, respectively.
    Every intermediate layer $L_k$ $(k = 0, \dotsc, {\alpha(Q_0)}^2)$ consists of ${\alpha(Q_0)}^2$ vertices $v_{k,p,q}$ $(p,q \in [{\alpha(Q_0)}])$.
    The source $s$ connects to $v_{0,p,p}$ for $p \in [\alpha(Q_0)]$ with the label $1$. 
    For $k = 1, \dotsc, {\alpha(Q_0)}^2$ and $p, q \in [{\alpha(Q_0)}]$, the incoming arcs to $v_{l,p,q}$ are those from $v_{k-1,p,r}$ with labels $A_{rq}$ for $r \in [{\alpha(Q_0)}]$, where the row (column) set of $A$ is identified with $[\alpha(Q_0)]$.
    We can inductively check that the polynomial computed at $v_{k,p,q}$ is the $(p,q)$ entry in $A^k$.
    We further add an extra vertex $v^*_k$ to $L_k$ for $k = 2, \dotsc, {\alpha(Q_0)}^2$ and connect from $v_{k-1,p,p}$ to $v^*_k$ with label $1$ for every $p \in [\alpha(Q_0)]$.
    We also draw arcs from $v_{{\alpha(Q_0)}^2,p,p}$ to $t$ for all $p$ in the same way.
    Then, $v^*_k$ computes $\tr A^{k-1}$ and $t$ computes $\tr A^{{\alpha(Q_0)}^2}$.
    These polynomials can be aggregated to $t$ by appending arcs from $v^*_k$ to $v^*_{k+1}$ for $k = 2, \dotsc, {\alpha(Q_0)}^2 - 1$ with label 1 and from $v^*_{{\alpha(Q_0)}^2}$ to $t$ with label 1.
    This completes the proof.
\end{proof}

By \cref{thm:abp-identity-test} and \cref{lem:into-abp}, we obtain the following.

\begin{theorem}
    We can check the semistability of a representation $V$ of a quiver $Q = (Q_0, Q_1)$ in $O({\alpha(Q_0)}^{20} + {\alpha(Q_0)}^{2\omega - 2} |Q_1|)$ time, where $\alpha = \dimv V$ and $\omega$ denotes the exponential of the complexity of matrix multiplication.
\end{theorem}

\begin{proof}
    Let $A_{ij} \coloneqq \Out(i) \cap \In(j)$ be the set of arcs from $i \in Q_0$ to $j \in Q_0$.
    To reduce the number of arcs, we first find $B_{ij} \subseteq A_{ij}$ such that $\{V(a) : a \in B_{ij}\}$ is a base of $\spnn{\{V(a) : a \in A_{ij}\}}$ for each $i, j \in Q_0$.
    Let $Q' = (Q_0, Q'_1)$ with $Q'_1 \coloneqq \cup_{i,j \in Q_0} B_{ij}$ and $V'$ a representation of $Q'$ naturally obtained from $V$ by restricting the arc set to $Q'_1$.
    We show that $V$ is semistable on $Q$ if and only if $V'$ is on $Q'$.
    By \cref{thm:bruyn-procesi}, it suffices to show that $\tr V(C) \ne 0$ for some closed path $C$ of $Q$ if and only if $\tr V'(C') = \tr V(C') \ne 0$ for some closed path $C'$ of $Q'$.
    The ``if'' part is clear as $C'$ is also a closed path of $Q$.
    To show the ``only if'' part, suppose that $\tr V(C') = 0$ for all closed paths $C'$ of $Q'$.
    This means that, for any sequence $i_0, i_1, \dotsc, i_{k-1}, i_k = i_0 \in Q_0$ of vertices,
    \begin{align}\label{eq:trace-reduce-arcs}
        \tr \left[ \left( \sum_{a \in B_{i_0 i_1}} x_a V(a) \right)
        \left( \sum_{a \in B_{i_1 i_2}} x_a V(a) \right)
        \dotsm
        \left( \sum_{a \in B_{i_{k-1} i_0}} x_a V(a) \right) \right]
    \end{align}
    is a zero polynomial.
    Therefore, replacing the $l$th factor in~\eqref{eq:trace-reduce-arcs} with any linear combination of $V(a)$ over $a \in B_{i_{l-1} i_l}$ for each $l \in [k]$ cannot make the trace nonzero; hence, $\tr V(C) = 0$ holds for all closed paths $C$ of $Q$.

    Through Gaussian elimination with fast matrix multiplication, the quiver $Q'$ can be obtained in $O(|A_{ij}|{(\alpha(i)\alpha(j))}^{\omega - 1}) = O(|A_{ij}|{\alpha(Q_0)}^{2\omega - 2})$ time for each $i, j$ and $O(|Q_1| {\alpha(Q_0)}^{2\omega - 2})$ time in total. 
    We then construct the ABP for $Q'$ promised by \cref{lem:into-abp} and apply \cref{thm:abp-identity-test} to the zero testing.
    Since the size of the ABP is $O({\alpha(Q_0)}^4)$ and the number of variables, which is the number of arcs in $Q'$, is at most $\sum_{i,j \in Q_0} \alpha(i) \alpha(j) = {\alpha(Q_0)}^2$, the running time of the zero testing is $O({\alpha(Q_0)}^{20})$.
\end{proof}

\section*{Acknowledgments}
The authors thank Hiroshi Hirai and Keiya Sakabe for their valuable comments on an earlier version of this paper.
The last author thanks Cole Franks for bringing the submodularity of quiver representations to his attention.
The authors thank an anonymous reviewer for pointing out the reference~\cite{Mulmuley2017}.
This work was supported by JSPS KAKENHI Grant Number JP24K21315, Japan.
The first author was supported by JSPS KAKENHI Grant Numbers JP22K17854, JP24K02901, Japan.
The second author was supported by JST, ERATO Grant Number JPMJER1903, JST, CREST Grant Number JPMJCR24Q2, JST, FOREST Grant Number JPMJFR232L, and JSPS KAKENHI Grant Number JP22K17853, Japan.
The last author was supported by JPSP KAKENHI Grant Number JP19K20212, and JST, PRESTO Grant Number JPMJPR24K5, Japan.

\printbibliography

\newpage
\appendix
\section{Elementary proof of King's criterion}\label{sec:King}
Here, we provide an elementary proof of King's criterion.
In \cref{sec:king-ones}, we start from the base case where $\alpha$ is the all-one vector, i.e., $V(a) \in \C$ for all arcs $a \in Q_1$.

\subsection{The case when $\alpha = \ones$}\label{sec:king-ones}

The following lemma characterizes the $\sigma$-semistability of a representation with $\alpha = \ones$, which has already been mentioned in \Cref{subsec:Intro:rank-one}.
Here, we show this via linear programming.
Recall that the support quiver of a representation $V$ of a quiver $Q$ means the subquiver of $Q$ whose arc set is $\supp(V) \coloneqq \{a \in Q_1 : V(a) \ne 0\}$.

\begin{lemma}\label{lem:1d-King}
    Let $V$ be a representation of $Q$ with dimension vector $\alpha = \ones$ and $\sigma$ a weight.
    Then, the following are equivalent.
    \begin{enumerate}[{label={\textup{(\arabic*)}}}]
        \item $V$ is $\sigma$-semistable.
        \item $\inf_{x \in \R^{Q_0}} h_{V, \sigma}(x) > 0$, where
            \begin{align}\label{def:f_V}
               h_{V, \sigma}(x) \coloneqq \sum_{a \in Q_1} \abs{V(a)}^2 e^{x(ha) - x(ta)} + \exp\left(\sum_{i \in Q_0}\sigma(i)x(i) \right) \quad (x \in \R^{Q_0}).
            \end{align}
        \item There exists an integral flow $\varphi \in \Z_+^{\supp(V)}$ with $\partial\varphi = \sigma$ on the support quiver of $V$.
        \item $\sigma(Q_0)=0$ and $\sigma(X) \leq 0$ for each lower set $X$ of the support quiver of $V$.
    \end{enumerate}
\end{lemma}
\begin{proof}
    By definition, $V$ is $\sigma$-semistable if and only if
    \begin{align}
    &\inf_{g \in \GL(Q, \alpha)} \left(\sum_{a \in Q_1} \abs{g_{ha}V(a)g_{ta}^{-1}}^2 + \abs{\chi_\sigma(g)}^2 \right) \\
    &= \inf_{g \in \GL(Q, \alpha)} \left(\sum_{a \in Q_1} \abs{V(a)}^2 \abs{g_{ha}}^2 \abs{g_{ta}}^{-2} + \prod_{i \in Q_0}\abs{g_i}^{2\sigma(i)} \right) > 0.
    \end{align}
    Letting $\abs{g_i}^2 = e^{x(i)}$, where $x(i) \in \R$, we can rewrite this as $\inf_{x \in \R^{Q_0}} h_{V, \sigma}(x) > 0$.
    This is equivalent to the optimal value of the LP
    \begin{alignat}{2}
        \min & \quad \sum_{i \in Q_0} \sigma(i) x(i) \\
        \text{s.t.} & \quad x(ha) - x(ta) \leq -1 \quad (a \in Q_1, V(a) \neq 0)
    \end{alignat}
    being not equal to $-\infty$.
    Note that the above LP is feasible because $Q$ is acyclic.
    By the LP duality theorem, this is equivalent to the existence of a feasible solution of the dual LP
    \begin{alignat}{2}
        \max & \quad - \sum_{a \in \supp(V)} \varphi(a) \\
        \text{s.t.} & \quad \partial \varphi = \sigma \\
        & \quad \varphi \in \R_+^{\supp(V)}.
    \end{alignat}
    In particular, we can equivalently replace $\varphi \in \R_+^{\supp(V)}$ with $\varphi \in \Z_+^{\supp(V)}$ since the constraint $\partial\varphi = \sigma$ is totally dual integral.
    By Gale's theorem~\cite{Gale1957} (see, e.g., \cite[Theorem~9.2]{KorteVygen2018}), this is equivalent to $\sigma(Q_0)=0$ and $\sigma(X) \leq 0$ for each lower set $X$ of the support quiver of $V$.
\end{proof}

For later use, we give a lower bound on $\inf_{x \in \R^{Q_0}} h_{V, \sigma}(x)$ that is \emph{continuous} in $V$.
Let $\Phi^\sigma \subseteq \Z_+^{Q_1}$ be the set of all integral flows on $Q$ with $\partial \varphi = \sigma$.
Note that $\Phi^\sigma$ is a finite set because $Q$ is acyclic.

\begin{lemma}\label{lem:h-lower}
    If $\Phi^\sigma \ne \emptyset$, we have
    \begin{align}
        \inf_{x \in \R^{Q_0}} h_{V, \sigma}(x) \ge \frac{1}{|\Phi^\sigma|} \sum_{\varphi \in \Phi^\sigma} \prod_{a \in \supp(\varphi)} \left( \frac{|V(a)|^2}{\varphi(a)} \right)^{\frac{\varphi(a)}{1 + \norm{\varphi}_1}}.
    \end{align}
\end{lemma}

\begin{proof}
    We fix $\varphi \in \Phi^\sigma$ and show
    \begin{align}
        \inf_{x \in \R^{Q_0}} h_{V, \sigma}(x) \ge \prod_{a \in \supp(\varphi)} \left( \frac{|V(a)|^2}{\varphi(a)} \right)^{\frac{\varphi(a)}{1 + \norm{\varphi}_1}}.
    \end{align}
    If $V(a) = 0$ for some $a \in \supp(\varphi)$, the bound trivially holds as $h_{V, \sigma}(x) \ge 0$.
    Suppose $V(a) \ne 0$ for all $a \in \supp(\varphi)$.
    By $\partial\varphi = \sigma$, we have
    \begin{align}
        \sum_{i \in Q_0} \sigma(i)x(i)
        = \sum_{i \in Q_0} (\partial \varphi)(i) x(i)
        = \sum_{a \in Q_1} \varphi(a) (x(ta) - x(ha))
    \end{align}
    for $x \in \R^{Q_0}$.
    Thus,
    \begin{align}
        \inf_{x \in \R^{Q_0}} h_{V, \sigma}(x)
        &= \inf_{x \in \R^{Q_0}} \left(
            \sum_{a \in Q_1} \abs{V(a)}^2 e^{x(ha) - x(ta)}
            + \exp\left(\sum_{a \in Q_1} \varphi(a) (x(ta) - x(ha))\right)
        \right) \\
        &\ge \inf_{y \in \R^{Q_1}} \left(
            \sum_{a \in Q_1} \abs{V(a)}^2 e^{-y(a)}
            + \exp\left(\sum_{a \in Q_1} \varphi(a) y(a)\right)
        \right) \\
        &= \inf_{y \in \R^{\supp(\varphi)}} \left(
            \sum_{a \in \supp(\varphi)} \abs{V(a)}^2 e^{-y(a)}
            + \exp\left(\sum_{a \in \supp(\varphi)} \varphi(a) y(a)\right)
        \right).\label{eq:f_V_lower}
    \end{align}
    As~\eqref{eq:f_V_lower} is the minimization of a strictly convex function, the unique stationary point $y^* \in \R^{\supp(\varphi)}$ attains the minimum if it exists.
    Letting
    \begin{align}
        g(y) \coloneqq \exp\left(\sum_{a \in \supp(\varphi)} \varphi(a) y(a)\right) = \prod_{a \in \supp(\varphi)} e^{\varphi(a) y(a)},
    \end{align}
    we can write the first-order optimality criterion for $y^*$ as $-|V(a)|^2 e^{-y^*(a)} + \varphi(a) g(y^*) = 0$ for $a \in \supp(\varphi)$.
    Thus, we have
    \begin{align}
        y^*(a) = \log \frac{|V(a)|^2}{\varphi(a)g(y^*)}.
    \end{align}
    Substituting this back to $g(y)$, we obtain
    \begin{gather}
        g(y^*)
        = \prod_{a \in \supp(\varphi)} \left( \frac{|V(a)|^2}{\varphi(a)g(y^*)} \right)^{\varphi(a)}
        = g(y^*)^{-\norm{\varphi}_1} \prod_{a \in \supp(\varphi)} \left( \frac{|V(a)|^2}{\varphi(a)} \right)^{\varphi(a)}
    \shortintertext{and}
        g(y^*) = \prod_{a \in \supp(\varphi)} \left( \frac{|V(a)|^2}{\varphi(a)} \right)^{\frac{\varphi(a)}{1 + \norm{\varphi}_1}},
    \end{gather}
    implying the desired lower bound as~\eqref{eq:f_V_lower} is at least $g(y^*)$.
\end{proof}

\subsection{General case}

Now, let us move on to the general case to show King's criterion. 
Let $V$ be a representation with the dimension vector $\alpha$.
By definition, $V$ is $\sigma$-semistable if and only if
\begin{align}
    &\inf_{g \in \GL(Q, \alpha)} 
    \left( \sum_{a \in Q_1} \norm*{g_{ha}V(a)g_{ta}^{-1}}_F^2 + \abs{\chi_\sigma(g)}^2 \right) \\
    &=\inf_{g \in \GL(Q, \alpha)}\left( \sum_{a \in Q_1} \tr(V(a)^\dagger g_{ha}^\dagger g_{ha} V(a) (g_{ta}^\dagger g_{ta})^{-1}) + \prod_{i \in Q_0} \det (g_i^\dagger g_i)^{\sigma(i)} \right) \\
    &=\inf_{X \in \PD(Q, \alpha)}\left( \sum_{a \in Q_1} \tr(V(a)^\dagger X_{ha} V(a) X_{ta}^{-1}) + \prod_{i \in Q_0} \det X_i^{\sigma(i)} \right) > 0.
\end{align}
Here, $\PD(Q, \alpha) \coloneqq \prod_{i \in Q_0} \PD(\alpha(i))$, in which $\PD(n)$ denotes the set of positive definite matrices of degree $n$.
Let us take an eigendecomposition of $X_i$ as
\begin{align}
    X_i = U_i \Diag(e^{x_i}) U_i^\dagger = \sum_{j \in [\alpha(i)]} e^{x_{i}(j)} \bu_{i,j} \bu_{i,j}^\dagger,
\end{align}
where $x_i \in \R^{\alpha(i)}$, $U_i \in U(\alpha(i))$, and $\bu_{i,j}$ is the $j$th column of $U_i$.
Then, the objective function inside the infimum becomes
\begin{align}
    f(x, U) \coloneqq \sum_{a \in Q_1} \sum_{j \in [\alpha(ha)]} \sum_{k \in [\alpha(ta)]} {\bigl| \bu_{ha,j}^\dagger V(a) \bu_{ta,k} \bigr|}^2 e^{x_{ha}(j) - x_{ta}(k)} + \exp\left(\sum_{i \in Q_0} \sigma(i) \sum_{j \in [\alpha(i)]} x_{i}(j) \right).
\end{align}

Observe that $f(x, U)$ for a fixed $U$ is in the form of the objective function~\eqref{def:f_V} of the base case. 
More precisely, let $Q' = (Q_0', Q_1')$ be a quiver such that each vertex $i \in Q_0$ is copied into $\alpha(i)$ copies $(i, j)$ for $j \in [\alpha(i)]$ and each arc $a \in Q_1$ is copied $\alpha(ha)\alpha(ta)$ times to connect the copies of the original endpoints.
Then, $V|_U = \bigl\{\bu_{ha,j}^\dagger V(a) \bu_{ta,k} : a \in Q_1,\ j \in [\alpha(ha)],\ k \in [\alpha(ta)]\bigr\}$ is a representation of $Q'$ with all-one dimension vector.
Let $\sigma'$ be a weight on $Q'$ such that $\sigma'(i,j) = \sigma(i)$.
Then, we have $f(x, U) = h_{V|_U, \sigma'}(x)$.

Since $\inf_{x,U} f(x, U) > 0$ implies $\inf_{x} f(x, U) > 0$ for any fixed unitary $U$, by \cref{lem:1d-King}, the $\sigma$-semistability of $V$ implies the following property:
\begin{itemize}
    \item[(P)] for any unitary $U$, $\sigma'(Q_0) = 0$ and $\sigma'(X) \leq 0$ for each lower set $X$ of the support quiver of $V|_U$.
\end{itemize}
To show the converse direction, suppose that (P) holds.
By \cref{lem:1d-King}, this implies the existence of $\varphi \in \Phi^{\sigma'}$ with $\supp(\varphi) \subseteq \supp(V|_U)$ for any $U$.
By \cref{lem:h-lower}, $\inf_{x} f(x, U)$ is lower bounded by some positive-valued continuous function of $U$.
Since the direct sum of the unitary groups is compact, we have that $\inf_{x, U} f(x, U) > 0$, i.e., $V$ is $\sigma$-semistable.

We finally show that (P) is equivalent to King's criterion.
First, note that
\begin{align}
    \sigma'(Q_0') = \sum_{i \in Q_0} \sigma(i) \alpha(i) = \sigma(\alpha) = \sigma(\dimv V).
\end{align}
Thus, $\sigma(\dimv V) = 0$ if and only if $\sigma'(Q_0') = 0$.
Let $X$ be a lower set of the support quiver of $V|_U$.
Then, $W(i) = \spnn{ \{\bu_{i,j} : (i, j) \in X \}}$ ($i \in Q_0$) defines a subrepresentation of $V$ and $\sigma(\dimv W) = \sigma'(X) \leq 0$.
Conversely, for a subrepresentation $W$, take a unitary $U$ such that $W(i)$ is spanned by $\bu_{i, j}$ for $j \in [\dim W(i)]$.
Define $X \subseteq Q_0'$ as $X = \{ (i, j) \in Q_0' : \bu_{i,j} \in W(i) \}$.
Then, $X$ is a lower set of the support quiver of $V|_U$ and $\sigma'(X) = \sigma(\dimv W) \leq 0$.
This completes the proof of King's criterion.

\end{document}